\documentclass[10pt,a4paper,listof=totoc,bibliography=totoc,listof=flat]{article}

\usepackage{amssymb}
\usepackage{amsmath}
\usepackage{amsthm}
\usepackage{authblk}
\usepackage{tikz}
\usepackage{subfig}
\usepackage{caption}
\usepackage{float}
\usepackage{mathrsfs}
\usepackage{eufrak}

\newtheorem{theorem}{Theorem}
\newtheorem{lemma}{Lemma}

\newtheorem{proposition}{Proposition}

\newtheorem{example}{Example}

\newtheorem*{theorem*}{Theorem}

\newcommand{\CD}{\mathcal{CD}}
\newcommand{\supply}{\mathfrak{s}}
\newcommand{\demand}{\mathfrak{d}}
\newtheorem*{conjecture}{Conjecture}
\newcommand{\R}{\mathbb{R}}

\begin{document}
\title{The Hierarchy of Circuit Diameters and Transportation Polytopes}

\author{S. Borgwardt, J. A. De Loera, E. Finhold, J. Miller}

\maketitle

\begin{abstract}
The study of the diameter of the graph of polyhedra is a classical problem in the theory of linear programming. While transportation polytopes are at the core of operations research and statistics it is still open whether the Hirsch conjecture is true for general $m{\times}n$--transportation polytopes.  In earlier work the first three authors introduced a  hierarchy of variations to the notion of graph diameter in polyhedra. The key reason was that  this hierarchy provides some interesting lower bounds for the usual graph diameter. 

This paper has three contributions: First, we  compare the hierarchy of diameters for the $m{\times}n$--transportation polytopes. We show that the Hirsch conjecture bound of $m+n-1$ is actually valid in most of these diameter notions. Second, we prove that for $3{\times}n$--transportation polytopes the Hirsch conjecture holds in the classical graph diameter. Third, we show for $2{\times}n$--transportation polytopes that the stronger monotone Hirsch conjecture holds and improve earlier bounds on the graph diameter.
\end{abstract}

\section{Introduction}

Transportation problems are among the oldest and most fundamental problems in mathematical programming, operations research, and statistics \cite{dk-13,dkos-09,h-41,h-63,kw-68,kyk-84}. An $m{\times}n$--transportation problem has $m$ supply points and $n$ demand points to be met. Each supply point holds a quantity $u_i>0$ and each demand point needs a quantity $v_j>0$. If $y_{ij}\geq 0$ decribes the flow from the supply point $i$ to the demand point $j$, then the set of feasible flow assignments, $y\in\mathbb{R}^{m\times n}$, can be described by
$$
\begin{array}{lcrclcl}
 &  & \sum\limits_{j=1}^{n}  y_{ij}    & =& u_i & \quad & i=1,...,m,\\
                         &      & \sum\limits_{i=1}^m y_{ij} & =    &  v_j
                         & \quad & j=1,...,n,\\
                         &      &          y_{ij}
                         & \geq  & 0                   & \quad &  i=1,...,m,\text{ }j=1,...,n.\\
\end{array}
$$ The set of solutions to these constraints constitutes a {\em transportation polytope}. Here the vectors $u$ and $v$ are called the \emph{marginals} or \emph{margins} for the transportation polytope, and a point $y$ inside the polytope is called a \emph{feasible flow assignment}. 

The standard transportation problem requires the optimization of a linear objective function over this set. A common way to solve these problems is the simplex algorithm \cite{d-63}. In the context of a worst-case performance of the simplex method, the study of the {\em graph diameter} (or combinatorial diameter) of polyhedra is a classical field in the theory of linear programming. This is the diameter of its underlying \emph{$1$-skeleton}. Hence the {\em graph distance} between two vertices ($0$-faces) in $P$ is the minimum number of edges ($1$-faces) needed to go from one vertex to the other, and the graph diameter of $P$ is the maximum distance between its vertices. The connection to the simplex algorithm becomes even more direct when investigating the ``monotone diameter". Here the diameter is realized by a monotone path on the same graph for a given linear functional. This monotone path is an edge-walk visiting vertices whose objective function values are non-decreasing with respect to the functional.

In 1957, W. Hirsch stated his famous conjecture (e.g., \cite{d-63}) claiming that the diameter of a polytope is at most $f-d$, where $d$ is its dimension and $f$ its number of facets. In his recent celebrated work, Santos finally gave a counterexample for general polytopes \cite{s-11}, but  the Hirsch conjecture is true for some classes of polytopes. A survey is found in \cite{ks-10}. In particular, it holds for dual transportation polyhedra \cite{b-84} and for $0,1$-polytopes \cite{n-89}. For the latter, the even stronger monotone Hirsch conjecture (or monotonic bounded Hirsch conjecture), asking whether $f-d$ is a bound on the monotone diameter, was shown to be true \cite{mt-93}.
But it is still open whether the Hirsch conjecture holds for $m{\times}n$--transportation polytopes despite a long line of research which we outline next.

For $m{\times}n$--transportation polytopes, the Hirsch conjecture  states an upper bound of $m+n-1$. This bound holds for $m=2$ \cite{dk-13} and a much lower bound than the one claimed in the Hirsch conjecture holds for the so-called partition polytopes, a special class of $0,1$-transportation polytopes \cite{b-11}. This is a generalization of the well-known fact that the assignment polytope has diameter $2$ for $n\geq 4$ \cite{br-74}. For $m\geq 3$, the best known  bounds for $m{\times}n$--transportation polytopes are in fact linear: a bound of $8(m+n-2)$ is presented in \cite{bhs-06}, and this is improved to $4(m+n-1)$ (it remains unpublished but a sketch of the proof is shown in \cite{dk-13}). 

In an attempt to understand the behavior of the graph diameter we introduced a hierarchy of distances and diameters for polyhedra that extend the usual edge walk \cite{bfh-14,blf-14}: Instead of only going along actual edges of the polyhedron, we walk along \emph{circuits}, which are all potential edge directions of the polyhedron. This means in particular that we can possibly enter the interior of the polyhedron.

The transportation polytopes are of the form $P=\{\,z\in \R^d:Az=b,\,z\geq 0\,\}$ for an integral matrix $A$. 
Then the \emph{circuits} or \emph{elementary vectors} associated with $A$ are those vectors  $g\in \ker(A) \setminus \{\,0\,\}$, for which $g$ is support-minimal in the set $\ker(A)\backslash \{0\}$, where $g$ is normalized to (coprime) integer components. These are exactly all the edge directions of $\{\,z\in \R^n:Az=b,z\geq 0\,\}$ that appear when letting $b$ vary. For vertices $v^{(1)},v^{(2)}$ of $P$, we call a sequence $v^{(1)}=y^{(0)},\ldots,y^{(k)}=v^{(2)}$ a \emph{circuit walk of length $k$}, if for all $i=0,\ldots,k-1$ we have $y^{(i+1)}-y^{(i)}=\alpha_ig^{i}$ for some  circuit $g^{i}$ and some $\alpha_i>0$. The \emph{circuit distance} from $v^{(1)}$ to $v^{(2)}$ is the minimum length of a circuit walk from $v^{(1)}$ to $v^{(2)}$. The \emph{circuit diameter} of $P$ is the maximum circuit distance between any two vertices of $P$. 
In the following we prove lower and upper bounds on the circuit diameters of transportation polytopes using various notions of diameters. We here look at different notions of circuit diameter that arise by putting further restrictions on the circuit walks as introduced in \cite{blf-14}; we consider four main types of circuit distance: 
\begin{itemize}
	\item $\CD$ (the circuit walks do not have to satisfy any additional properties), 
	\item $\CD_f$ (all points in the circuit walk have to be feasible points in the polyhedron),
	\item $\CD_{fm}$ (maximal feasible steps are applied, that is, for all $i$, $y^{(i)}+\alpha_i g^{i}\in P$, but $y^{(i)}+\alpha g^{i}\notin P$ for all $\alpha>\alpha_i$; this is the circuit distance introduced in \cite{bfh-14}) ,
	\item $\CD_{e}$ (the circuit steps go along the edges from vertex to vertex; this distance corresponds to the graph diameter and was denoted $\CD_{efm}$ in \cite{blf-14}).
\end{itemize}
These four ways to measure the distance between vertices form the `central chain' in the larger hierarchy of distances shown in \cite{blf-14}. Note that they satisfy the relation 
	\[
			\CD_e\geq \CD_{fm}\geq \CD_f\geq \CD\; ,
	\]
by the simple fact that the definitions become less restrictive. 

We are interested in the diameters of transportation polytopes with respect to these notions of distances. The subscript we use in $\CD_*$ changes according with the distance being employed. For simplicity of notation, we use $\CD_*$ with slightly different meanings, which are clear in the context of the presentation. It can mean
\begin{itemize}
\item the diameter of a specific transportation polytope, 
\item the maximal diameter of any polytope in a given family of transportation polytopes, 
\item or the actual distance between two specific vertices within a polytope.
\end{itemize}

Hereby, the above hierarchy directly translates to a hierarchy of diameters for a given polytope, as well as a hierarchy of maximal diameters in a family of transportation polytopes.

Our first results in this paper are lower and upper bounds on the `bottom' part of this hierarchy of diameters.



\begin{theorem}\label{thm:generalTP}
For all $m,n$, there exist $m{\times}n$-transportation polytopes with $\CD \geq \min\{(m-1)(n-1),m+n-1\}.$
Since for all $m{\times}n$--transportation polytopes $\CD_f$ is bounded above by $m+n-1$, the Hirsch bound holds with respect to $\CD_f$ for all transportation polytopes. Further, the Hirsch bound is attained in at least one example for all $m\geq 3, n\geq 4$ with respect to $\CD$.
\end{theorem}

Essentially, Theorem \ref{thm:generalTP} means that these weaker diameter concepts satisfy the bound imposed by the Hirsch conjecture and that it is tight in the sense that there exist margins for which it is realized. In the literature this is often called {\em  Hirsch-sharp} (see \cite{hk-98,hk-99}). Of course, there also are margins for which the diameter is even lower \cite{br-74,b-11}. Theorem \ref{thm:generalTP} indicates that, one way for disproving the Hirsch conjecture for transportation polytopes is to find a counterexample with respect to the distance notion $\CD_{fm}$.

The later sections of this paper are dedicated to studying how the upper bounds in the hierarchy collapse for transportation polytopes with low number of supply points.  As we will see, the hierarchy collapses fully for the $3{\times}n$--case, and we keep a (tight) gap of one between $\CD_e$ and $\CD_{fm}$ for the class of $2{\times}n$--polytopes.

We prove this by bounding the diameters of $2{\times}n$-- and $3{\times}n$-transportation polytopes. We begin with the discussion of $\CD_{fm}$ for $2{\times}n$-transportation polytopes.

\begin{theorem}\label{thm:CD_fm}
 All $2{\times}n$-transportation polytopes satisfy $\CD_{fm}\leq n-1$.
\end{theorem}
We then turn to the classical graph diameters of these polytopes, i.e. $\CD_{e}$. First, we refine the upper bound  $n+1$  on the diameter of $2{\times}n$--transportation polytopes \cite{dk-13} (which is the one implied by the Hirsch conjecture) by one and prove that this bound is realized by a monotone path. 

\begin{theorem}\label{thm:2n}
The monotone Hirsch conjecture holds for $2{\times}n$--transportation polytopes, with a (tight) upper bound of $n$ on the diameter.
\end{theorem}

At first glance, the bound of $n$ looks like a minor refinement, but there are three important differences here:  First, we are not Hirsch-sharp anymore. Second, our new approach allows us to prove the stronger monotone Hirsch conjecture. And third, the bound above is tight in the sense that there exist $2{\times n}$-transportation polytopes with margins such that $\CD_e=n$. For these polytopes we have in particular $\CD_{fm}\leq n-1<n=\CD_e$.

Further, our approach serves as an introduction to a `marking system' before we continue on to the more involved $3{\times}n$ case. It is a key ingredient in showing the following. 

\begin{theorem}\label{thm:3n}
The Hirsch conjecture holds for $3{\times}n$--transportation polytopes. In particular, they have graph diameter $\leq n+2$.
\end{theorem}

Throughout the paper, we will reveal some concepts that hold for general $m\times n$-transportation polytopes and we hope will prove helpful in investigating the diameters of transportation polytopes for larger $m$.

The paper is structured as follows: In Section $2$, we recall the necessary background on transportation polytopes and present our notation and tools for the discussion. In Section $3$, we turn to the hierarchy of diameters of general transportation polytopes and prove that and how its `lower parts' collapse (Theorem \ref{thm:generalTP}).  Section $4$ is dedicated to bounding the circuit diameter $\CD_{fm}$ for $2{\times}n$-transportation polytopes (Theorem \ref{thm:CD_fm}). In Section $5$, we then turn to the graph diameter of $2{\times}n$-- and $3{\times}n$--transportation polytopes. We begin by discussing a  marking system for the basic variables of vertices on an edge walk. This marking system is the key ingredient in showing that in the $2{\times}n$--case we get a bound of $n$ and that a corresponding edge-walk is a monotone path (Theorem \ref{thm:2n}). Finally, we also use it to prove validity of the Hirsch conjecture for $3{\times}n$--transportation polytopes (Theorem \ref{thm:3n}).

\section{Preliminaries}

When discussing an $m{\times}n$--transportation problem, it is common practice to think of the supply and demand points as nodes in the complete bipartite graph $K_{m,n}$. We denote the nodes corresponding to the supply points $\{\supply_1,\ldots,\supply_m\}$ and the nodes corresponding to the demand points $\{\demand_1,\ldots,\demand_n\}$. For every feasible solution $y\in \mathbb{R}^{m\times n}$ we define the \emph{support graph} $B(y)$ as the subgraph of $K_{m,n}$ with edges $\left\{\{\supply_i,\demand_j\}\ :\  y_{ij}>0,i\in [m], j\in [n] \right\}$ of non-zero flow. We use this representation throughout the paper to visualize our methods. 

For general transportation polytopes this graph is not necessarily connected, but it is the case for non-degenerate transportation polytopes: An $m{\times}n$ transportation polytope is \emph{non-degenerate} if every vertex has exactly $m+n-1$ non-negative entries. This is the case if and only if there are no non-empty proper subsets $M\subsetneq [m]$ and $N\subsetneq [n]$ such that $\sum_{i\in M} u_i=\sum_{j\in N} v_j$, see \cite{kyk-84}.
Note that for each degenerate $m{\times}n$--transportation polytope, there is a non-degenerate $m{\times}n$--transportation polytope of the same or larger graph diameter $\CD_e$ \cite{kyk-84}. Therefore, it suffices to consider non-degenerate transportation polytopes to prove upper bounds on $\CD_e$. We exploit this in Section $4$. 

In contrast  we cannot assume non-degeneracy when exhibiting other notions of circuit distance, as it is not clear whether for every degenerate $m{\times}n$--transportation polytope there is a perturbed non-degenerate $m{\times}n$--transportation polytope bounding the respective circuit diameters of the original one above. 

When studying circuit distances, the vertices of the polytope are of special interest. They can be characterized by their support graphs: A feasible point $y$ is a vertex if and only if its support graph contains no cycles, that is, $B(y)$ is a spanning forest. In particular the vertices $y$ of non-degenerate transportation polytopes are given by spanning trees (see for example \cite{kw-68}), i.e{.} a vertex $y$ is uniquely determined by (the edge set of) its support graph $B(y)$.

Even though $B(y)$ is directly derived from $y$ (for all $y$, not only vertices), throughout this paper we use the term \emph{assignment} to refer to the tuple $(y,B(y))$ -- for our purposes, an assignment is two things at a time:
\begin{itemize}
\item a vector $y\in \mathbb{R}^{m\times n}$ and
\item a set of edges in $K_{m,n}$ inducing a support graph.
\end{itemize}
We typically use the capital letters $O$ (for `original'), $C$ (for `current'), and $F$ (for `final') to refer to such assignments. In the above sense, assignments can be vertices or lie in the interior of the polytope, we can count their number of edges, e.g{.} $|O|$, and so on. When we have to refer explicitly to the (actual) flow assignment corresponding to $O$, we do so by $y^O=(y_{11}^O,\dots,y_{mn}^O)$. 





For an assignment $O$ we distinguish two kinds of demand nodes:
\begin{enumerate}
	\item \emph{leaf demands} are those demand points which are leaves in $B(y^O)$. When we say \emph{leaf edges} we refer to edges incident to leaf demands. (This differs from the standard notion of leaf edges used in graph theory)
	\item \emph{mixed demands} are those demand points which have degree at least two in $B(y^O)$. We denote by $D_m^O$ the set of mixed demands and $E_m^O:=\left\{\{\supply_i,\demand_j\}: \demand_j\in D_m^O\right\}$, the \emph{mixed edges}.
\end{enumerate}

Note that for a vertex $O$ of a non-degenerate $2{\times}n$--transportation polytope we always have $|D_m^O|=1$ and $|E_m^O|=2$, for the non-degenerate $3{\times}n$--case either $|D_m^O|=1$ and $|E_m^O|=3$ or $|D_m^O|=2$ and $|E_m^O|=4$, as illustrated in the assignments in Figure \ref{fig:mixeddemands}; mixed edges are bold. Here the sets $D_m^O$ of mixed demands are  $\{\demand_4\}$, $\{\demand_5\}$  and $\{\demand_4,\demand_6\}$, respectively.
\begin{figure}[H]
	\center
			\begin{tikzpicture}[scale=0.4]
			
					\coordinate (c1) at (0,5);
					\coordinate (c2) at (0,3);
					\coordinate (x0) at	(5,7);	
					\coordinate (x1) at	(5,6);	
					\coordinate (x2) at (5,5);							
					\coordinate (x3) at (5,4);
					\coordinate (x4) at (5,3);
					\coordinate (x5) at (5,2);
					\coordinate (x6) at (5,1);
 					\draw [fill, black] (c1) circle (0.1cm);
					\draw [fill, black] (c2) circle (0.1cm);
					\draw [fill, black] (x0) circle (0.1cm);
					\draw [fill, black] (x1) circle (0.1cm);
					\draw [fill, black] (x2) circle (0.1cm);
					\draw [fill, black] (x3) circle (0.1cm);
					\draw [fill, black] (x4) circle (0.1cm);
					\draw [fill, black] (x5) circle (0.1cm);
					\draw [fill, black] (x6) circle (0.1cm);
					\node[left] at (c1) {$\supply_1$};
					\node[left] at (c2) {$\supply_2$};
					\node[right] at (x0) {$\demand_1$};
					\node[right] at (x1) {$\demand_2$};
					\node[right] at (x2) {$\demand_3$};
					\node[right] at (x3) {$\demand_4$};
					\node[right] at (x4) {$\demand_5$};
					\node[right] at (x5) {$\demand_6$};
					\node[right] at (x6) {$\demand_7$};
					\draw (c1)--(x1);
					\draw[very thick] (c1)--(x3);
				  \draw[very thick] (c2)--(x3);
					\draw (c1)--(x2);
					\draw (c1)--(x0);
				  \draw (c2)--(x4);		
				  \draw (c2)--(x5);
				  \draw (c2)--(x6);					  

					\coordinate (c1) at (10,6);
					\coordinate (c2) at (10,4);
					\coordinate (c3) at (10,2);
					\coordinate (x1) at	(15,8);	
					\coordinate (x2) at (15,7);	
					\coordinate (x21) at (15,6);	
					\coordinate (x22) at (15,5);	
					\coordinate (x3) at (15,4);
					\coordinate (x42) at (15,3);
					\coordinate (x4) at (15,2);
					\coordinate (x5) at (15,1);
					\coordinate (x6) at (15,0);
					\draw [fill, black] (c1) circle (0.1cm);
					\draw [fill, black] (c2) circle (0.1cm);
					\draw [fill, black] (c3) circle (0.1cm);
					\draw [fill, black] (x1) circle (0.1cm);
					\draw [fill, black] (x2) circle (0.1cm);
					\draw [fill, black] (x3) circle (0.1cm);
					\draw [fill, black] (x4) circle (0.1cm);
					\draw [fill, black] (x5) circle (0.1cm);
					\draw [fill, black] (x6) circle (0.1cm);
					\draw [fill, black] (x21) circle (0.1cm);
					\draw [fill, black] (x22) circle (0.1cm);
					\draw [fill, black] (x42) circle (0.1cm);
					\node[left] at (c1) {$\supply_1$};
					\node[left] at (c2) {$\supply_2$};
					\node[left] at (c3) {$\supply_3$};
					\node[right] at (x1) {$\demand_1$};
					\node[right] at (x2) {$\demand_2$};
					\node[right] at (x21) {$\demand_3$};
					\node[right] at (x22) {$\demand_4$};
					\node[right] at (x3) {$\demand_5$};
					\node[right] at (x42) {$\demand_6$};
					\node[right] at (x4) {$\demand_7$};
					\node[right] at (x5) {$\demand_8$};
					\node[right] at (x6) {$\demand_9$};
					\draw (c1)--(x1);
					\draw[very thick] (c1)--(x3);
				  \draw[very thick] (c2)--(x3);
				  \draw[very thick] (c3)--(x3);
					\draw (c1)--(x2);
					\draw (c2)--(x22);
					\draw (c2)--(x21);
					\draw (c2)--(x42);
				  \draw (c3)--(x4);		
				  \draw (c3)--(x5);
				  \draw (c3)--(x6);
			
					\coordinate (c1) at (20,6);
					\coordinate (c2) at (20,4);
					\coordinate (c3) at (20,2);
					\coordinate (x0) at	(25,8);	
					\coordinate (x1) at	(25,6);	
					\coordinate (x2) at (25,7);							
					\coordinate (x3) at (25,5);
					\coordinate (x32) at (25,4);
					\coordinate (x33) at (25,3);
					\coordinate (x4) at (25,2);
					\coordinate (x5) at (25,1);
					\coordinate (x6) at (25,0);
					\draw [fill, black] (c1) circle (0.1cm);
					\draw [fill, black] (c2) circle (0.1cm);
					\draw [fill, black] (c3) circle (0.1cm);
					\draw [fill, black] (x0) circle (0.1cm);
					\draw [fill, black] (x1) circle (0.1cm);
					\draw [fill, black] (x2) circle (0.1cm);
					\draw [fill, black] (x3) circle (0.1cm);
					\draw [fill, black] (x32) circle (0.1cm);
					\draw [fill, black] (x33) circle (0.1cm);
					\draw [fill, black] (x4) circle (0.1cm);
					\draw [fill, black] (x5) circle (0.1cm);
					\draw [fill, black] (x6) circle (0.1cm);
					\node[left] at (c1) {$\supply_1$};
					\node[left] at (c2) {$\supply_2$};
					\node[left] at (c3) {$\supply_3$};
					\node[right] at (x0) {$\demand_1$};
					\node[right] at (x1) {$\demand_2$};
					\node[right] at (x2) {$\demand_3$};
					\node[right] at (x3) {$\demand_4$};
					\node[right] at (x32) {$\demand_5$};
					\node[right] at (x33) {$\demand_6$};
					\node[right] at (x4) {$\demand_7$};
					\node[right] at (x5) {$\demand_8$};
					\node[right] at (x6) {$\demand_9$};
					\draw (c1)--(x1);
					\draw[very thick] (c1)--(x3);
				  \draw[very thick] (c2)--(x3);
				  \draw[very thick] (c2)--(x33);
				  \draw[very thick] (c3)--(x33);
					\draw (c1)--(x0);
					\draw (c2)--(x32);
					\draw (c2)--(x33);
					\draw (c1)--(x2);		
				  \draw (c3)--(x4);
				  \draw (c3)--(x5);
				  \draw (c3)--(x6);					  				  
			\end{tikzpicture}				  
	\caption{Vertices of non-degenerate $2{\times}n$--transportation polytopes always have exactly one mixed demand point (left), while vertices of non-degenerate $3{\times}n$--transportation polytopes can have either one or two (middle, right).}\label{fig:mixeddemands}
\end{figure}
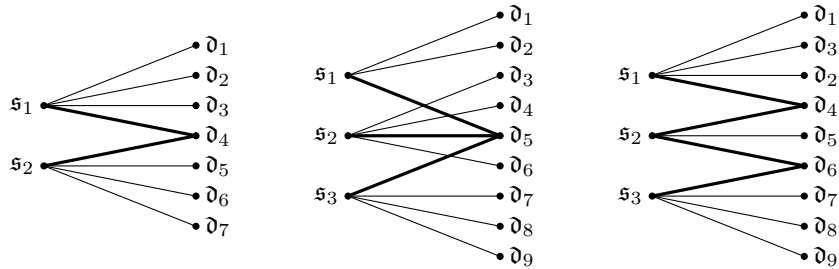
For every supply point $\supply_i$ we denote the vertices adjacent to $\supply_i$ (its neighbourhood) in an assignment $O$ by 
\[N^O_i:=\left\{\demand_j:y^O_{ij}>0\right\}=\left\{\demand_j:\{\supply_i,\demand_j\}\in O\right\}\ .\] 

We continue with characterizing the \emph{actual} edges in terms of the support graphs:
\begin{proposition}[see \cite{kw-68}]\label{prop2}
Let $O$ and $C$ be two vertices of an $m{\times}n$--transportation polytope. Then they are connected by an edge if and only if $O\cup C$ contains a unique cycle.
\end{proposition}
This unique cycle describes an edge direction of the transportation polytope. It is easy to see that every cycle of $K_{m,n}$ can appear as an edge of some $m{\times}n$--transportation polytope if we choose suitable margins. Thus, the set of circuits of an $m{\times}n$--transportation polytope just consists of all even simple cycles of the form $(\supply_{i_1},\demand_{j_1},\supply_{i_2},\demand_{j_2},\ldots, \supply_{i_k},\demand_{j_k})$.
Applying such a circuit 
at a (feasible) point $y$ corresponds changing the flow on the edges of $K_{m,n}$: We increase flow on all edges $\{\supply_{i_l},\demand_{j_l}\}$ and decrease flow on all edges $\{\demand_{i_l},\supply_{j_{l+1}}\}$ by the same arbitrary amount, the {\emph step length}. (For a shorter wording, we will often say that we \emph{increase} or \emph{decrease} edges.)

This in particular ensures that the `margin equations' defining the transportation polytope remain satisfied. However, we do not necessarily remain feasible, as applying as circuit possibly decreases the flow on an edge $\{\demand_{i_l},\supply_{j_{l+1}}\}$  below its lower bound of zero. So if we want to remain feasible, the step length at $y$ can be at most the minimum over all $y_{j_l,i_{l+1}}$.
This implies that the circuit steps with respect to the four different concepts are as follows: 
\begin{itemize}
	\item $\CD$: We can apply any circuit with any step length.
	\item $\CD_f$: We can apply any circuit $(\supply_{i_1},\demand_{j_1},\supply_{i_2},\demand_{j_2},\ldots, \supply_{i_k},\demand_{j_k})$ for which $y_{j_li_{l+1}}>0$ for all $l$. The step length is at most the minimum over all $y_{j_l,i_{l+1}}$.
	\item $\CD_{fm}$: We can apply any circuit $(\supply_{i_1},\demand_{j_1},\supply_{i_2},\demand_{j_2},\ldots, \supply_{i_k},\demand_{j_k})$ for which $y_{j_li_{l+1}}>0$ for all $l$. The step length equals the minimum over all $y_{j_l,i_{l+1}}$.
	\item $\CD_e$: By Proposition \ref{prop2}, an edge step (\emph{pivot}) goes from a vertex $O$ to a vertex $C$ that differs from $O$ in exactly one edge. This can be achieved by inserting an arbitrary edge $\{\supply_i,\demand_j\}\notin O$ in $O$. This closes an even cycle, which describes the circuit (respectively edge direction) we apply. Again we alternatingly increase and decrease along this cycle by the minimum over all $y_{j_l,i_{l+1}}$. Due to non-degeneracy (which we can assume here) this deletes exactly one edge and hence leads to a neighboring vertex $C$.
\end{itemize}

Observe that for $\CD_e$ every circuit step inserts one edge and deletes one edge and hence the corresponding support graphs are always cycle free. In contrast to this, for $\CD_{fm}$ we can insert multiple edges while deleting at least one edge, such that there can be cycles. In circuit walks of type $\CD_{f}$ we can insert multiple edges and we do not have to delete an edge at all. Finally, infeasible points can appear in walks of type $\CD$. This is why we do not consider a support graph in this case.



\medskip

For sake of notation, we distinguish two types of distances from an assignment $O$ to a fixed assignment $F$. We will use $\CD^O$, $\CD_{e}^O$, etc{.} to denote the respective circuits distances from $O$ and $F$, while the \emph{edge distance} $|O\backslash F|$ is just the number of edges that are in $O$, but not in $F$. $O\backslash F$ consists of those edges that have \emph{to be deleted} (or are \emph{to delete}) when walking from $O$ to $F$.

Clearly $\CD_e^O \geq |O\backslash F|$: By applying a single pivot at an assignment $O$, one obtains an assignment $C$ which has at most one additional edge (the new, inserted one) in common with $F$. In contrast, we can have $\CD_{fm}<|O\backslash F|$, as we will see in Example \ref{THEex1}.

This example was first mentioned in \cite{bhs-06}. It illustrates a situation that is crucial for proving the Hirsch conjecture for $m=2$ and understanding the graph diameter.
\begin{example}\label{THEex1}
We first consider an edge walk from an assigment $O$ to an assignment $F$. The nodes are labeled with the margins, the edges with the current flow; the bold edges highlight the circuit we apply and the dashed edges are those we insert. 
	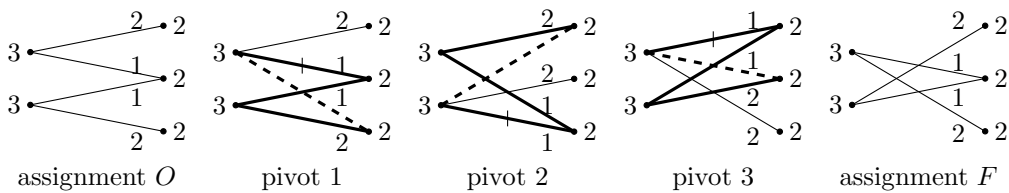
\begin{figure}[H]
		\centering
			\begin{tikzpicture}[scale=0.35]
					\coordinate (c1) at (0,4);
					\coordinate (c2) at (0,2);
					\coordinate (x1) at	(5,5);	
					\coordinate (x2) at (5,3);							
					\coordinate (x3) at (5,1);
					\draw [fill, black] (c1) circle (0.1cm);
					\draw [fill, black] (c2) circle (0.1cm);
					\draw [fill, black] (x1) circle (0.1cm);
					\draw [fill, black] (x2) circle (0.1cm);
					\draw [fill, black] (x3) circle (0.1cm);
					\node[left] at (c1) {$3$};
					\node[left] at (c2) {$3$};
					\node[right] at (x1) {$2$};
					\node[right] at (x2) {$2$};
				  \node[right] at (x3) {$2$};
				  \draw (c1)--(x2);
				  \draw (c2)--(x2);
					\draw (c1)--(x1);
				  \draw (c2)--(x3);			  
				  \node at (4,5.2) {$2$};
					\node at (4,3.5) {$1$};
					\node at (4,2.3) {$1$};
					\node at (4,0.6) {$2$};					
					\node at (2.5,-0.8) {assignment $O$};			
			\end{tikzpicture}
			\begin{tikzpicture}[scale=0.35]
					\coordinate (c1) at (0,4);
					\coordinate (c2) at (0,2);
					\coordinate (x1) at	(5,5);	
					\coordinate (x2) at (5,3);							
					\coordinate (x3) at (5,1);
					\draw [fill, black] (c1) circle (0.1cm);
					\draw [fill, black] (c2) circle (0.1cm);
					\draw [fill, black] (x1) circle (0.1cm);
					\draw [fill, black] (x2) circle (0.1cm);
					\draw [fill, black] (x3) circle (0.1cm);
					\node[left] at (c1) {$3$};
					\node[left] at (c2) {$3$};
					\node[right] at (x1) {$2$};
					\node[right] at (x2) {$2$};
				  \node[right] at (x3) {$2$};
				  \draw[very thick]  (c1)--(x2);
				  \draw[very thick]  (c2)--(x2);
				  \draw (2.5,3.8)--(2.5,3.2);
					\draw (c1)--(x1);
				  \draw[very thick]  (c2)--(x3);
				  \draw[dashed, very thick] (c1)--(x3);				  
				  \node at (4,5.2) {$2$};
					\node at (4,3.5) {$1$};
					\node at (4,2.3) {$1$};
					\node at (4,0.6) {$2$};	
					\node at (2.5,-0.8) {pivot $1$};										
			\end{tikzpicture}
			\begin{tikzpicture}[scale=0.35]
					\coordinate (c1) at (0,4);
					\coordinate (c2) at (0,2);
					\coordinate (x1) at	(5,5);	
					\coordinate (x2) at (5,3);							
					\coordinate (x3) at (5,1);
					\draw [fill, black] (c1) circle (0.1cm);
					\draw [fill, black] (c2) circle (0.1cm);
					\draw [fill, black] (x1) circle (0.1cm);
					\draw [fill, black] (x2) circle (0.1cm);
					\draw [fill, black] (x3) circle (0.1cm);
					\node[left] at (c1) {$3$};
					\node[left] at (c2) {$3$};
					\node[right] at (x1) {$2$};
					\node[right] at (x2) {$2$};
				  \node[right] at (x3) {$2$};
				  \draw[very thick]  (c1)--(x1);
				  \draw[dashed, very thick] (c2)--(x1);
				  \draw[very thick]  (c1)--(x3);
				  \draw (c2)--(x2);
					\draw[very thick]  (c2)--(x3);
					\draw (2.5,1.8)--(2.5,1.2);					
					\node at (4,5.3) {$2$};
					\node at (4,3.2) {$2$};
					\node at (4,2.0) {$1$};
					\node at (4,0.6) {$1$};
					\node at (2.5,-0.8) {pivot $2$};										
			\end{tikzpicture}
			\begin{tikzpicture}[scale=0.35]
					\coordinate (c1) at (0,4);
					\coordinate (c2) at (0,2);
					\coordinate (x1) at	(5,5);	
					\coordinate (x2) at (5,3);							
					\coordinate (x3) at (5,1);
					\draw [fill, black] (c1) circle (0.1cm);
					\draw [fill, black] (c2) circle (0.1cm);
					\draw [fill, black] (x1) circle (0.1cm);
					\draw [fill, black] (x2) circle (0.1cm);
					\draw [fill, black] (x3) circle (0.1cm);
					\node[left] at (c1) {$3$};
					\node[left] at (c2) {$3$};
					\node[right] at (x1) {$2$};
					\node[right] at (x2) {$2$};
				  \node[right] at (x3) {$2$};
				  \draw[very thick]  (c1)--(x1);
				  \draw (c1)--(x3);
				  \draw[very thick]  (c2)--(x1);
					\draw[very thick]  (c2)--(x2);
					\draw[dashed, very thick] (c1)--(x2);
					\draw (2.5,4.8)--(2.5,4.2);
					\node at (4,5.2) {$1$};
					\node at (4,3.7) {$1$};
					\node at (4,2.3) {$2$};
					\node at (4,0.8) {$2$};
					\node at (2.5,-0.8) {pivot $3$};											
			\end{tikzpicture}
			\begin{tikzpicture}[scale=0.35]
					\coordinate (c1) at (0,4);
					\coordinate (c2) at (0,2);
					\coordinate (x1) at	(5,5);	
					\coordinate (x2) at (5,3);							
					\coordinate (x3) at (5,1);
					\draw [fill, black] (c1) circle (0.1cm);
					\draw [fill, black] (c2) circle (0.1cm);
					\draw [fill, black] (x1) circle (0.1cm);
					\draw [fill, black] (x2) circle (0.1cm);
					\draw [fill, black] (x3) circle (0.1cm);
					\node[left] at (c1) {$3$};
					\node[left] at (c2) {$3$};
					\node[right] at (x1) {$2$};
					\node[right] at (x2) {$2$};
				  \node[right] at (x3) {$2$};
				  \draw (c1)--(x3);
				  \draw (c1)--(x2);
				  \draw (c2)--(x1);
					\draw (c2)--(x2);					
					\node at (4,5.2) {$2$};
					\node at (4,3.6) {$1$};
					\node at (4,2.3) {$1$};
					\node at (4,0.9) {$2$};			
					\node at (2.5,-0.8) {assignment $F$};											
			\end{tikzpicture}
		\caption{An edge walk from vertex $O$ to vertex $F$ of length $3$.}
	\end{figure} 
In the first step, no matter which edge in $F\backslash O$ we insert, we have to delete an edge that is contained in $F$. Hence this is a walk of minimum length and thus $\CD_e^O$ is strictly larger than $|O\backslash F|$. 

In contrast, we can go from $O$ to $F$ in only one circuit step of type $\CD_{fm}$ (and thus also $\CD_f$ or $\CD$):  As we allow to go through the interior of the polytope, we can insert and delete two edges in just one step.
	\begin{figure}[H]
		\centering
			\begin{tikzpicture}[scale=0.35]
					\coordinate (c1) at (0,4);
					\coordinate (c2) at (0,2);
					\coordinate (x1) at	(5,5);	
					\coordinate (x2) at (5,3);							
					\coordinate (x3) at (5,1);
					\draw [fill, black] (c1) circle (0.1cm);
					\draw [fill, black] (c2) circle (0.1cm);
					\draw [fill, black] (x1) circle (0.1cm);
					\draw [fill, black] (x2) circle (0.1cm);
					\draw [fill, black] (x3) circle (0.1cm);
					\node[left] at (c1) {$3$};
					\node[left] at (c2) {$3$};
					\node[right] at (x1) {$2$};
					\node[right] at (x2) {$2$};
				  \node[right] at (x3) {$2$};
				  \draw (c1)--(x2);
				  \draw (c2)--(x2);
					\draw (c1)--(x1);
				  \draw (c2)--(x3);			  
				  \node at (4,5.2) {$2$};
					\node at (4,3.5) {$1$};
					\node at (4,2.3) {$1$};
					\node at (4,0.6) {$2$};					
					\node at (2.5,-0.8) {assignment $O$};			
			\end{tikzpicture}
			\begin{tikzpicture}[scale=0.35]
					\coordinate (c1) at (0,4);
					\coordinate (c2) at (0,2);
					\coordinate (x1) at	(5,5);	
					\coordinate (x2) at (5,3);							
					\coordinate (x3) at (5,1);
					\draw [fill, black] (c1) circle (0.1cm);
					\draw [fill, black] (c2) circle (0.1cm);
					\draw [fill, black] (x1) circle (0.1cm);
					\draw [fill, black] (x2) circle (0.1cm);
					\draw [fill, black] (x3) circle (0.1cm);
					\node[left] at (c1) {$3$};
					\node[left] at (c2) {$3$};
					\node[right] at (x1) {$2$};
					\node[right] at (x2) {$2$};
				  \node[right] at (x3) {$2$};
				  \draw (c1)--(x2);
				  \draw (c2)--(x2);
					\draw (2.5,1.8)--(2.5,1.2);					
					\draw (2.5,4.8)--(2.5,4.2);
					\draw[very thick]  (c1)--(x1);
				  \draw[very thick]  (c2)--(x3);
				  \draw[dashed, very thick] (c1)--(x3);				 
				  \draw[dashed, very thick] (c2)--(x1);	 
				  \node at (4,5.2) {$2$};
					\node at (4,3.5) {$1$};
					\node at (4,2.3) {$1$};
					\node at (4,0.6) {$2$};	
					\node at (2.5,-0.8) {circuit step $1$};										
			\end{tikzpicture}
			\begin{tikzpicture}[scale=0.35]
					\coordinate (c1) at (0,4);
					\coordinate (c2) at (0,2);
					\coordinate (x1) at	(5,5);	
					\coordinate (x2) at (5,3);							
					\coordinate (x3) at (5,1);
					\draw [fill, black] (c1) circle (0.1cm);
					\draw [fill, black] (c2) circle (0.1cm);
					\draw [fill, black] (x1) circle (0.1cm);
					\draw [fill, black] (x2) circle (0.1cm);
					\draw [fill, black] (x3) circle (0.1cm);
					\node[left] at (c1) {$3$};
					\node[left] at (c2) {$3$};
					\node[right] at (x1) {$2$};
					\node[right] at (x2) {$2$};
				  \node[right] at (x3) {$2$};
				  \draw (c1)--(x3);
				  \draw (c1)--(x2);
				  \draw (c2)--(x1);
					\draw (c2)--(x2);					
					\node at (4,5.2) {$2$};
					\node at (4,3.6) {$1$};
					\node at (4,2.3) {$1$};
					\node at (4,0.9) {$2$};			
					\node at (2.5,-0.8) {assignment $F$};											
			\end{tikzpicture}
		\caption{A feasible maximal circuit walk from vertex $O$ to vertex $F$ of length $1$.}
	\end{figure}
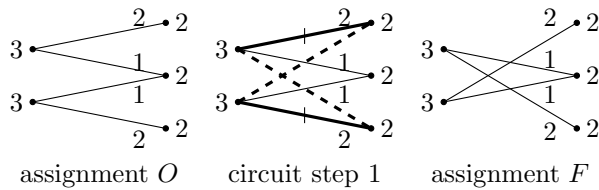 
\end{example}

\medskip

Finally, observe that the diameter bound of $m+n-1$, which we will prove in Section $4$, is closely related to the Hirsch-bound.
We know that the dimension of an $m{\times}n$--transportation polytope equals $(m-1)(n-1)$ \cite{kw-68}, and for non-degenerate transportation polytopes the number of facets is equal to $m\cdot n - k$, where $k$ is the number of \emph{critical edges}. This follows immediately from Theorem 2 in \cite{kw-68}. We call an edge $\{\supply_i,\demand_j\}$ \emph{critical} for a transportation polytope, if the edge $\{\supply_i,\demand_j\}$ exists in every support graph, that is, $y_{ij}>0$ for all solutions $y$. With this we can state the Hirsch conjecture as
\begin{conjecture}[Hirsch conjecture for transportation polytopes]
The graph diameter of an $m{\times}n$--tranportation polytope is at most $m+n-1-k$, where $k$ is the number of critical edges of the transportation polytope.
\end{conjecture}
Observe that $m+n-1-k$ is exactly the number of edges in which two assignments of the corresponding transportation polytope can differ. In particular, for proving the Hirsch conjecture for $m{\times}n$--transportation polytopes it is enough to show that there is a sequence of pivots that inserts the edges in $F$ such that no inserted edge is deleted again. 
\section{The hierarchy of diameters for general transportation polytopes}

In this section, we investigate upper and lower bounds on the chain of diameters for the hierarchy
$$ \CD_e \geq \CD_{fm} \geq \CD_f \geq \CD.$$
Before turning to the proof of Theorem \ref{thm:generalTP}, let us recall which bounds are currently known for the graph diameter of transportation polytopes. These are the upper and lower bounds that are currently known.

\begin{proposition}\label{prop:general}
\noindent\begin{itemize}
\item All $m{\times}n$-transportation polytopes satisfy $\CD_{e}\leq 8(m+n-2)$ \cite{bhs-06}.
\item  For all $m,n$, there exist $m{\times}n$-transportation polytopes with $$\CD_{e}\geq \min\{(m-1)(n-1),m+n-1\}.$$
\end{itemize}
\end{proposition}
The latter is easy to see. There are assignments $O$ and $F$ that differ by exactly $\min\{(m-1)(n-1),m+n-1\}$ edges, as that is the rank of the underlying constraint matrix. The lower bound on $\CD_e$ then follows from being able to insert only one single edge from $F\backslash O$ into a current assignment at a time. Recall that $(m-1)(n-1)$ is the dimension of the polytope. We have $m+n-1\leq (m-1)(n-1)$ for all $m\leq n$ but $m=2$ or $m=n=3$.

Theorem \ref{thm:generalTP} states two new bounds on circuit diameters that are intimately related to the ones above. We split the proof into two parts. First, we show that $\CD_f\leq m+n-1$ for all transportation polytopes. Essentially, the `lower part' of the hierarchy satisfies the Hirsch bound. Second, we will show that for all combinations of $m,n$, there exist $m{\times}n$--transportation polytopes with $\CD \geq \min\{(m-1)(n-1),m+n-1\}$, which generalizes the lower bound in Proposition \ref{prop:general} to all of the hierarchy.

For these proofs, we have to introduce the notion of distance $\CD_s$ using a so-called \emph{sign-compatible} circuit walk (denoted $\CD_{fs}$ in \cite{blf-14}). Two vectors $x$ and $y$ are sign-compatible (with respect to the identity matrix) if $x$ and $y$ belong to the same orthant of $\R^d$. For a polytope of the form $P=\{\,z\in \R^d:Az=b,\,z\geq 0\,\}$, such as a transportation polytope, and two vertices $v^{(1)},v^{(2)}$ a circuit walk is sign-compatible if all circuits used are pairwise sign-compatible and are sign-compatible with $v^{(2)}-v^{(1)}$. Note that a sign-compatible walk is in particular feasible and thus $\CD_s \geq \CD_f$ \cite{blf-14}.

We perform the proof by showing that the more restrictive category $\CD_{s}$ satisfies $\CD_{s} \leq m+n-1$.

\begin{lemma}\label{lem:upperboundonCDf}
All $m{\times}n$--transportation polytopes satisfy $\CD_f \leq m+n-1$.
\end{lemma}

\begin{proof}
Recall that in a standard representation of a transportation polytope $P=\{y\in \mathbb{R}^{m\times n}: Ay = \binom uv , y \geq 0\}$, the matrix $A\in \mathbb{Z}^{(m+ n)\times (m\cdot n)}$ has row rank $m+n-1$. In other words, of the $m+n$ margin equalities, one is redundant and can be derived from the others. We obtain $\text{rank}(A)=m+n-1$. Then $P$ satisfies $\CD_f \leq \CD_{s}\leq m+n-1$ by Corollary $2$ in \cite{blf-14}.
\end{proof}

Second, we show the existence of margins for which $\CD \geq \min\{(m-1)(n-1),m+n-1\}$. We do so by outlining a general principle of constructing such a set of margins in the corresponding transportation polytope. 

\begin{lemma}\label{lem:perturb}
For all $m,n$, there exist $m{\times}n$--transportation polytopes with $$\CD \geq \min\{(m-1)(n-1),m+n-1\}.$$
\end{lemma}

\begin{proof}
We begin by constructing an $m{\times}n$--transportation polytope $P$ with margins $u,v$ and two vertices $O,F$ such that there is a sign-compatible circuit walk from $O$ to $F$ that uses exactly $k=\min\{(m-1)(n-1),m+n-1\}$ linearly independent circuits. 

Consider the set of circuits depicted in Figure \ref{fig:whichcircuitstouse}, where the dashed edges are the ones being increased. It is not difficult to check that for all pairs $m\leq n$, there exist at least $k$ such circuits, they are all linearly independent, sign-compatible, and that neither the union of the edges to decrease (the solid edges) nor the union of the edges to increase (the dashed edges) contains a cycle. 


	\begin{figure}[H]
		\centering  
\subfloat[$2\leq j\leq n$  if $n\leq 3$ $\quad $ $2\leq j\leq n{-}1$ if $n\geq 4$]{
			\begin{tikzpicture}[scale=0.35]
					\coordinate (c1) at (0,4);
					\coordinate (c2) at (0,2);
					\coordinate (c3) at (0,0);
					\coordinate (x1) at	(5,4);	
					\coordinate (x2) at (5,2);							
					\coordinate (x3) at (5,0);
					\draw [fill, black] (c1) circle (0.1cm);
					\draw [fill, black] (c2) circle (0.1cm);
					\draw [fill, white] (c3) circle (0.1cm);
					\draw [fill, black] (x1) circle (0.1cm);
					\draw [fill, black] (x2) circle (0.1cm);
					\draw [fill, white] (x3) circle (0.1cm);
					\node[left] at (c1) {$\supply_{1}$};
					\node[left] at (c2) {$\supply_2$};
					\node[left,below] at (c3) {$ $};
					\node[right] at (x1) {$\demand_{1}$};
					\node[right] at (x2) {$\demand_{j}$};
				 	 \draw (c1)--(x1);
				  	\draw[dashed] (c2)--(x1);
				  	\draw[dashed] (c1)--(x2);
				 	 \draw (c2)--(x2);							
			\end{tikzpicture}
}
\hspace*{0.5cm}\subfloat[$3\leq i\leq m$]{
			\begin{tikzpicture}[scale=0.35]
					\coordinate (c1) at (0,4);
					\coordinate (c2) at (0,2);
					\coordinate (c3) at (0,0);
					\coordinate (x1) at	(5,4);	
					\coordinate (x2) at (5,2);							
					\coordinate (x3) at (5,0);
					\draw [fill, black] (c1) circle (0.1cm);
					\draw [fill, black] (c2) circle (0.1cm);
					\draw [fill, white] (c3) circle (0.1cm);
					\draw [fill, black] (x1) circle (0.1cm);
					\draw [fill, black] (x2) circle (0.1cm);
					\draw [fill, white] (x3) circle (0.1cm);
					\node[left] at (c1) {$\supply_{1}$};
					\node[left] at (c3) {$\supply_i$};
					\node[right] at (x1) {$\demand_{1}$};
					\node[right] at (x2) {$\demand_2$};
				 	 \draw (c1)--(x1);
				  	\draw[dashed] (c3)--(x1);
				  	\draw[dashed] (c1)--(x2);
				 	 \draw (c3)--(x2);							
			\end{tikzpicture}
}

\hspace*{0.5cm}\subfloat[if $m\geq 3$]{
			\begin{tikzpicture}[scale=0.35]
					\coordinate (c1) at (0,4);
					\coordinate (c2) at (0,2);
					\coordinate (c3) at (0,0);
					\coordinate (x1) at	(5,4);	
					\coordinate (x2) at (5,2);							
					\coordinate (x3) at (5,0);
					\draw [fill, black] (c1) circle (0.1cm);
					\draw [fill, black] (c2) circle (0.1cm);
					\draw [fill, black] (c3) circle (0.1cm);
					\draw [fill, black] (x1) circle (0.1cm);
					\draw [fill, black] (x2) circle (0.1cm);
					\draw [fill, black] (x3) circle (0.1cm);
					\node[left] at (c1) {$\supply_1$};
					\node[left] at (c2) {$\supply_2$};
					\node[left] at (c3) {$\supply_3$};
					\node[right] at (x1) {$\demand_1$};
					\node[right] at (x2) {$\demand_2$};
					\node[right] at (x3) {$\demand_3$};
				 	 \draw (c1)--(x1);
				  	\draw[dashed] (c2)--(x1);
				  	\draw[dashed] (c1)--(x2);
					\draw[dashed] (c3)--(x3);
				 	 \draw (c2)--(x3);	
 					\draw (c3)--(x2);							
			\end{tikzpicture}
}
\subfloat[if $n \geq 4$]{
			\begin{tikzpicture}[scale=0.35]
					\coordinate (c1) at (0,4);
					\coordinate (c2) at (0,2);
					\coordinate (c3) at (0,0);
					\coordinate (x1) at	(5,4);	
					\coordinate (x2) at (5,2);							
					\coordinate (x3) at (5,0);
					\draw [fill, black] (c1) circle (0.1cm);
					\draw [fill, black] (c2) circle (0.1cm);
					\draw [fill, white] (c3) circle (0.1cm);
					\draw [fill, black] (x1) circle (0.1cm);
					\draw [fill, black] (x2) circle (0.1cm);
					\draw [fill, white] (x3) circle (0.1cm);
					\node[left] at (c1) {$\supply_{1}$};
					\node[left] at (c2) {$\supply_2$};
					\node[left,below] at (c3) {$ $};
					\node[right] at (x2) {$\demand_{2}$};
					\node[right] at (x3) {$\demand_{n}$};
				 	 \draw[dashed] (c1)--(x2);
				  	\draw (c2)--(x2);
				  	\draw[dashed] (c2)--(x3);
				 	 \draw (c1)--(x3);							
			\end{tikzpicture}
}
\hspace*{0.5cm}\subfloat[if $m\geq 3$, $n \geq 4$]{
			\begin{tikzpicture}[scale=0.35]
					\coordinate (c1) at (0,4);
					\coordinate (c2) at (0,2);
					\coordinate (c3) at (0,0);
					\coordinate (x1) at	(5,4);	
					\coordinate (x2) at (5,2);							
					\coordinate (x3) at (5,0);
					\draw [fill, black] (c1) circle (0.1cm);
					\draw [fill, black] (c2) circle (0.1cm);
					\draw [fill, black] (c3) circle (0.1cm);
					\draw [fill, black] (x1) circle (0.1cm);
					\draw [fill, black] (x2) circle (0.1cm);
					\draw [fill, black] (x3) circle (0.1cm);
					\node[left] at (c1) {$\supply_1$};
					\node[left] at (c2) {$\supply_2$};
					\node[left] at (c3) {$\supply_3$};
					\node[right] at (x1) {$\demand_1$};
					\node[right] at (x2) {$\demand_2$};
					\node[right] at (x3) {$\demand_n$};
				 	 \draw (c1)--(x1);
				  	\draw[dashed] (c1)--(x2);
					\draw (c2)--(x2);	
					\draw[dashed] (c3)--(x1);
				 	 \draw[dashed] (c2)--(x3);	
 					\draw (c3)--(x3);							
			\end{tikzpicture}
}
		\caption{A set of sign-compatible, linearly independent circuits. }\label{fig:whichcircuitstouse}
	\end{figure}
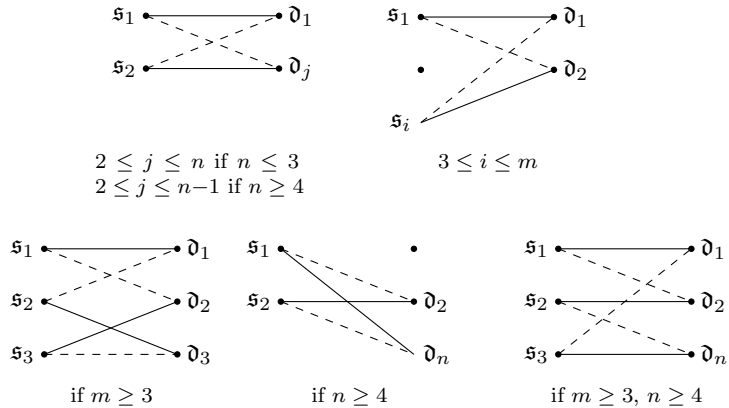 

Given a set of $k$ such circuits $g^i$, define the margins $u,v$ of $P$ to componentwisely be the number of circuits in this set that the respective point lies on. Then the set of edges to decrease induce a vertex $O$ as they do not form a cycle, and similarly the set of edges to increase induce a vertex $F$. 

Note this construction implies that $y^F-y^O=\sum\limits_{i=1}^{k} g^i$. Let us now define a new polytope $P'$ by perturbing the margins w.l.o.g `along  $g^1$': This means that we derive new margins $u',v'$ from $u,v$ by choosing an $\epsilon>0$ and then setting $u_j'=u_j+\epsilon$ and  $v_j'=v_j+\epsilon$ for all supply and demand points incident with the circuit $g^1$. For $\epsilon$ sufficiently small, the same support graphs $B(y^O), B(y^F)$ still induce vertices $y'^O,y'^F$ of $P'$ and we see $y'^F-y'^O=(1+\epsilon) g^1+\sum\limits_{i=2}^{k} g^i$. Of course, such a perturbation can be done with respect to any subset of the $g^i$, so that each $g^i$ gets its own small $\epsilon_i$.  Further, since each of the $g^i$ are sign compatible, we may choose each $\epsilon_i$ independently. So we have a choice of margins such that the perturbed vectors
$$y^F-y^O=\sum\limits_{i=1}^{k} (1+\epsilon_i)g^i,$$
live in a $k$-dimensional space, for which $y'^O,y'^F$ both correspond to vertices of the respective polytopes $P'$.

As the $g^i$ are linearly independent, no strict subset suffices to be able to walk from $y^O$ to $y^F$ in $P$. Now consider a set of (up to) $k-1$ circuits $g'^1,\dots,g'^{k-1}$ and assume $y^O-y^F$ is in the linear subspace spanned by $g'^1,\dots,g'^{k-1}$. There is a $g^i$ which is linearly independent from $g'^1,\dots,g'^{k-1}$ and thus perturbing $P$ along such a $g^i$ yields an infinite set of polytopes $P'$ for which $y'^O-y'^F$ is not in the span of $g'^1,\dots,g'^{k-1}$. As there is only a finite number of sets of up to $k-1$ circuits (recall these only depend on the fixed matrix $A$, not the margins), there is a perturbation (or rather a sequence of perturbations) along the $g^i$ such that we obtain a polytope $P'$ for which $y'^F-y'^O$ is not in the linear subspace spanned by any set of (up to) $k-1$ circuits. This proves the claim.
\end{proof}
Let us demonstrate what a perturbation as described in Lemma \ref{lem:perturb} looks like in practice. Recall Example \ref{THEex1}, and the second circuit walk finishing in just a single step. We call this circuit $g$.

The $2{\times}3$--transportation polytope in the example has dimension $2$. Our goal is to come up with a perturbation of the margins such that any (not necessarily feasible) circuit walk from $O$ to $F$ has to use two steps. In particular we want to rule out going from $O$ to $F$ by only applying $g$.

To do so, consider the two circuits $g^1, g^2$ depicted in Figure \ref{fig:ex2circuits}. The dashed edges are the ones that are increased along the respective circuit. (Note that they are not sign-compatible, but this not relevant for our simple example, as we will only show a perturbation along a single circuit.)
	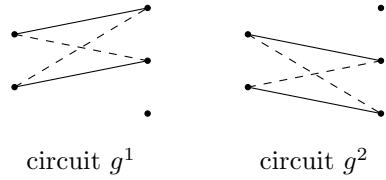
\begin{figure}[H]
		\centering
			\begin{tikzpicture}[scale=0.35]
					\coordinate (c1) at (0,4);
					\coordinate (c2) at (0,2);
					\coordinate (x1) at	(5,5);	
					\coordinate (x2) at (5,3);							
					\coordinate (x3) at (5,1);
					\draw [fill, black] (c1) circle (0.1cm);
					\draw [fill, black] (c2) circle (0.1cm);
					\draw [fill, black] (x1) circle (0.1cm);
					\draw [fill, black] (x2) circle (0.1cm);
					\draw [fill, black] (x3) circle (0.1cm);
				  \draw[dashed] (c1)--(x2);
				  \draw (c2)--(x2);
				  \draw (c1)--(x1);
				  \draw[dashed] (c2)--(x1);							
					\node at (2.5,-0.8) {circuit $g^1$};			
			\end{tikzpicture}
\hspace*{1cm}
			\begin{tikzpicture}[scale=0.35]
					\coordinate (c1) at (0,4);
					\coordinate (c2) at (0,2);
					\coordinate (x1) at	(5,5);	
					\coordinate (x2) at (5,3);							
					\coordinate (x3) at (5,1);
					\draw [fill, black] (c1) circle (0.1cm);
					\draw [fill, black] (c2) circle (0.1cm);
					\draw [fill, black] (x1) circle (0.1cm);
					\draw [fill, black] (x2) circle (0.1cm);
					\draw [fill, black] (x3) circle (0.1cm);
				  \draw (c1)--(x2);
				  \draw[dashed] (c2)--(x2);
				 \draw[dashed] (c1)--(x3);
				  \draw (c2)--(x3);
					\node at (2.5,-0.8){circuit $g^2$};										
			\end{tikzpicture}
		\caption{Two circuits $g^1,g^2$ in a $2{\times}3$--transportation polytope.}\label{fig:ex2circuits}
	\end{figure} 
They are linearly independent, as they both share only the edges $\{\supply_1,\demand_2\}, \{\supply_2,\demand_2\}$. Applying both of them with step length $2$ transfers $O$ to $F$ via a circuit walk of length $2$ (whose intermediate point is not feasible).

But we already know that one can go from $O$ to $F$ in just a single step by using $g$. So these margins do not satisfy the properties in Lemma \ref{lem:perturb}, and we have to apply the construction described in the second part of the proof. 

Note that $g$ is linearly independent both from $g^1$ and from $g^2$, so it does not matter which of the two circuits we pick for the construction. We use $g^1$ and choose a sufficiently small $\epsilon>0$. We then add $\epsilon$ to all nodes incident to $g^1$. This is depicted in Figure \ref{fig:perturbingacircuit}.

	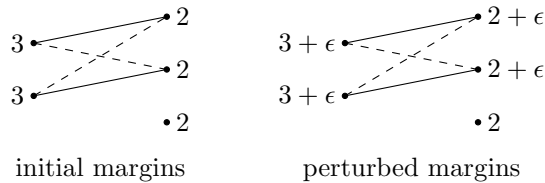
\begin{figure}[H]
		\centering
			\begin{tikzpicture}[scale=0.35]
					\coordinate (c1) at (0,4);
					\coordinate (c2) at (0,2);
					\coordinate (x1) at	(5,5);	
					\coordinate (x2) at (5,3);							
					\coordinate (x3) at (5,1);
					\draw [fill, black] (c1) circle (0.1cm);
					\draw [fill, black] (c2) circle (0.1cm);
					\draw [fill, black] (x1) circle (0.1cm);
					\draw [fill, black] (x2) circle (0.1cm);
					\draw [fill, black] (x3) circle (0.1cm);
  \draw[dashed] (c1)--(x2);
				  \draw (c2)--(x2);
				  \draw (c1)--(x1);
				  \draw[dashed] (c2)--(x1);	
					\node[left] at (c1) {$3$};
					\node[left] at (c2) {$3$};
					\node[right] at (x1) {$2$};
					\node[right] at (x2) {$2$};
				  \node[right] at (x3) {$2$};					
					\node at (2.5,-0.8) {initial margins};			
			\end{tikzpicture}
\hspace{2em}
			\begin{tikzpicture}[scale=0.35]
					\coordinate (c1) at (0,4);
					\coordinate (c2) at (0,2);
					\coordinate (x1) at	(5,5);	
					\coordinate (x2) at (5,3);							
					\coordinate (x3) at (5,1);
					\draw [fill, black] (c1) circle (0.1cm);
					\draw [fill, black] (c2) circle (0.1cm);
					\draw [fill, black] (x1) circle (0.1cm);
					\draw [fill, black] (x2) circle (0.1cm);
					\draw [fill, black] (x3) circle (0.1cm);
  \draw[dashed] (c1)--(x2);
				  \draw (c2)--(x2);
				  \draw (c1)--(x1);
				  \draw[dashed] (c2)--(x1);	
					\node[left] at (c1) {$3+\epsilon$};
					\node[left] at (c2) {$3+\epsilon$};
					\node[right] at (x1) {$2+\epsilon$};
					\node[right] at (x2) {$2+\epsilon$};
				  \node[right] at (x3) {$2$};				
					\node at (2.5,-0.8) {perturbed margins};											
			\end{tikzpicture}
		\caption{A perturbation of margins along circuit $g^1$.}\label{fig:perturbingacircuit}
	\end{figure} 

Then the original support graphs $B(y^O)$ and $B(y^F)$ again induce vertices $O'$ and $F'$ in the new polytope. We show the corresponding $y^{O'}$ and $y^{F'}$ in Figure \ref{fig:perturbedassignments}. Clearly an application of $g$ cannot transfer $O'$ to $F'$ now, as $y_{11}^{O'} \neq y_{23}^{O'}$.

	\begin{figure}[H]
		\centering
			\begin{tikzpicture}[scale=0.35]
					\coordinate (c1) at (0,4);
					\coordinate (c2) at (0,2);
					\coordinate (x1) at	(5,5);	
					\coordinate (x2) at (5,3);							
					\coordinate (x3) at (5,1);
					\draw [fill, black] (c1) circle (0.1cm);
					\draw [fill, black] (c2) circle (0.1cm);
					\draw [fill, black] (x1) circle (0.1cm);
					\draw [fill, black] (x2) circle (0.1cm);
					\draw [fill, black] (x3) circle (0.1cm);
					\node[left] at (c1) {$3+\epsilon$};
					\node[left] at (c2) {$3+\epsilon$};
					\node[right] at (x1) {$2+\epsilon$};
					\node[right] at (x2) {$2+\epsilon$};
				  \node[right] at (x3) {$2$};
				  \draw (c1)--(x2);
				  \draw (c2)--(x2);
					\draw (c1)--(x1);
				  \draw (c2)--(x3);			  
				  \node at (3.4,5.2) {\small $2+\epsilon$};
					\node at (4,3.5) {$1$};
					\node at (3.4,2.1) {\small $1+\epsilon$};
					\node at (4,0.6) {$2$};					
					\node at (2.5,-0.8) {assignment $O'$};			
			\end{tikzpicture}
			\begin{tikzpicture}[scale=0.35]
					\coordinate (c1) at (0,4);
					\coordinate (c2) at (0,2);
					\coordinate (x1) at	(5,5);	
					\coordinate (x2) at (5,3);							
					\coordinate (x3) at (5,1);
					\draw [fill, black] (c1) circle (0.1cm);
					\draw [fill, black] (c2) circle (0.1cm);
					\draw [fill, black] (x1) circle (0.1cm);
					\draw [fill, black] (x2) circle (0.1cm);
					\draw [fill, black] (x3) circle (0.1cm);
					\node[left] at (c1) {$3+\epsilon$};
					\node[left] at (c2) {$3+\epsilon$};
					\node[right] at (x1) {$2+\epsilon$};
					\node[right] at (x2) {$2+\epsilon$};
				  \node[right] at (x3) {$2$};
				  \draw (c1)--(x3);
				  \draw (c1)--(x2);
				  \draw (c2)--(x1);
					\draw (c2)--(x2);					
					\node at (3.4,5.2) {\small $2+\epsilon$};
					\node at (4,3.6) {\small $1+\epsilon$};
					\node at (4,2.3) {$1$};
					\node at (4,0.9) {$2$};			
					\node at (2.5,-0.8) {assignment $F'$};											
			\end{tikzpicture}
		\caption{The assignments corresponding to $B(y^O), B(y^F)$ in the perturbed polytope.}\label{fig:perturbedassignments}.
	\end{figure}
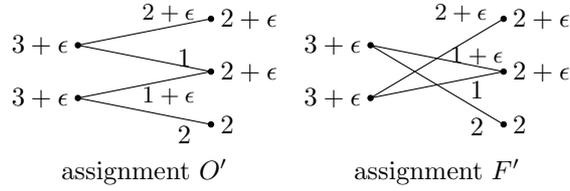 
By this perturbation, we ruled out the possibility of having a `shortcut' via circuit $g$. In a larger example, one would now continue with a sequence of perturbations of smaller and smaller perturbation values, e.g{.} $\epsilon,\epsilon^2,\epsilon^3,\dots$, to rule out any further shorter circuit walks . Choosing $\epsilon$ sufficiently small, one does not reintroduce a shorter circuit walk at a later point.

Recalling Theorem \ref{thm:generalTP}, it remains to study both $\CD_{fm}$ and $\CD_e$. In the next sections, we do so for $m=2$ and $m=3$ (and arbitrary $n$).
\section{The circuit diameter for $2{\times}n$--transportation polytopes}

This section is dedicated to proving Theorem \ref{thm:CD_fm} on the circuit diameter $\CD_{fm}$. We do not exclude the degenerate case as it is not clear whether for every degenerate transportation polytope there is a perturbed non-degenerate transportation polytope bounding the circuit diameter of the original one. Therefore, the support graphs are not necessarily connected, in particular the vertices are described by forests instead of spanning trees. To prove $\CD_{fm}\leq n-1$ for all $2{\times}n$--transportation polytopes, we show that, for an assignment $O$, it always is possible to delete an edge in $O\backslash F$ while only inserting edges in $F$. 

For $2{\times}n$--transportation polytopes, we can describe a circuit step by two disjoint edges in the current assignment on which we want to \textit{decrease} flow, one edge $\{\supply_1,\demand_2\}$ incident to $\supply_1$ and one edge $\{\supply_2,\demand_1\}$ incident to $\supply_2$. This implies that we increase on $\{\supply_1,\demand_1\}$ and $\{\supply_2,\demand_2\}$ (dashed). The latter edges are inserted if they do not exist in the current assignment.
	
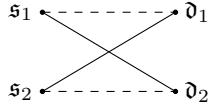
\begin{figure}[H]
		\centering
			\begin{tikzpicture}[scale=0.35]
					\coordinate (c1) at (0,3);
					\coordinate (c2) at (0,0);										
					\coordinate (xj) at (5,0);
					\coordinate (xl) at (5,3);
					\draw [fill, black] (c1) circle (0.08cm);
					\draw [fill, black] (c2) circle (0.08cm);
					\draw [fill, black] (xj) circle (0.08cm);
					\draw [fill, black] (xl) circle (0.08cm);
					\node[left] at (c1) {$\supply_1$};
					\node[left] at (c2) {$\supply_2$};
					\node[right] at (xj) {$\demand_2$};	
					\node[right] at (xl) {$\demand_1$};		
				  \draw (c1)--(xj);
				  \draw (c2)--(xl);
				  \draw[dashed] (c1)--(xl);
				  \draw[dashed] (c2)--(xj);
			\end{tikzpicture} 	
	\end{figure}

\begin{lemma}\label{lem:cd2n}
Let  $O$ and $F$ be two vertices of a $2{\times}n$--transportation polytope. Then the circuit distance from $O$ to $F$ is at most $|O\setminus F|$. Further if $|F|=n$, then the circuit distance from $O$ to $F$ is at most $|O\setminus F|-1$.
\end{lemma}

\begin{proof}
If $O\neq F$, there must be an edge in $O\backslash F$ that we have to delete.
We show that there is a circuit step that deletes such an edge in $O\backslash F$ and does not insert any edge not in $F$. 

\textbf{Case 1: } There are edges incident to both $\supply_1$ and $\supply_2$ to delete.\\ 
Let $\{\supply_1,\demand_j\}$ and $\{\supply_2,\demand_l\}$ be those edges. Apply the pivot that reduces flow on these particular edges. This deletes at least one of these edges and increases $\{\supply_1,\demand_l\}$ and $\{\supply_2,\demand_j\}$, both of which are in $F$.

\textbf{Case 2: } One supply node still has an edge to delete while the other does not. \\
Without loss of generality, let there be an edge incident to $\supply_2$ to delete, but no edges incident to $\supply_1$ to delete. Further let $\{\supply_2,\demand_l\}$ be this edge to delete. Then we have to increase $\{\supply_1,\demand_l\}$ and hence there is a mixed edge to decrease ($\{\supply_1,\demand_j\}$), but no edge to delete. Since $\supply_1$ may be incident to at most one mixed edge in $F$, this implies there are no other edges to decrease incident to $\supply_1$. We apply the pivot that decreases $\{\supply_1,\supply_j\}$ and $\{\supply_2,\demand_l\}$. Assume it would delete $\{\supply_1,\demand_j\}$. Then there would be no more edges incident to $\supply_1$ to decrease, but an edge to increase ($\{\supply_1,\demand_j\}$ would have to be inserted again). Hence $\{\supply_1,\demand_j\}$ is not deleted, and thus $\{\supply_2,\demand_l\}$ is the only edge that is deleted.\\

Note in both cases, an edge in $O\backslash F$ is deleted, and no edge in $F$ is ever deleted. So we will get that $|O\setminus F|=0$ after $|O\setminus F|$ of these circuit steps, at which point $O=F$.

Finally, if $|F|=n$, in order to show that the circuit distance is at most $|O\setminus F|-1$, we only need to show there is a circuit step that deletes two edges in $O\setminus F$ at once. Consider the last circuit step in our sequence of steps (dashed edges are increasing, solid are decreasing):
	
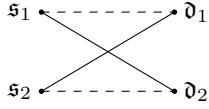
\begin{figure}[H]
		\centering
			\begin{tikzpicture}[scale=0.35]
					\coordinate (c1) at (0,3);
					\coordinate (c2) at (0,0);										
					\coordinate (xj) at (5,0);
					\coordinate (xl) at (5,3);
					\draw [fill, black] (c1) circle (0.08cm);
					\draw [fill, black] (c2) circle (0.08cm);
					\draw [fill, black] (xj) circle (0.08cm);
					\draw [fill, black] (xl) circle (0.08cm);
					\node[left] at (c1) {$\supply_1$};
					\node[left] at (c2) {$\supply_2$};
					\node[right] at (xj) {$\demand_2$};	
					\node[right] at (xl) {$\demand_1$};		
				  \draw (c1)--(xj);
				  \draw (c2)--(xl);
				  \draw[dashed] (c1)--(xl);
				  \draw[dashed] (c2)--(xj);
			\end{tikzpicture} 	
	\end{figure}

Suppose that one of the edges being decreased was not deleted (without loss of generality say $\{\supply_1,\demand_2\}$). Then since this is the last circuit step we have $\{\supply_1,\demand_2\},\{\supply_2,\demand_2\}\in F$. This implies that $F$ has a mixed demand, $\demand_2$. However, since $|F|=n$ and $F$ is a spanning tree, it cannot have any mixed demands since there must be exactly one edge incident to each $\demand_i$, thus we have a contradiction. This implies that in the last circuit step, two edges from $O\setminus F$ are deleted if $|F|=n$, and thus combining this with the above argument we have the circuit distance from $O$ to $F$ is $|O\setminus F|-1$ if $|F|=n$.
\end{proof}

Now $\CD_{fm}\leq n-1$ is a consequence of the following simple observation.

\begin{lemma}\label{lem0}
Let $O$ and $F$ be two vertices of a (possibly degenerate) $2{\times}n$--transportation polytope $P$. Then either $|O\setminus F|\leq n-1$ or $|O\setminus F|=n$ and $|F|=n$
\end{lemma}

\begin{proof}
Since $|O\setminus F|+|F|=|O\cup F|\leq 2n$, we have $|O\setminus F|\leq 2n-|F|\leq 2n-n=n$, since $|F|\geq n$ (a spanning forest in $K_{2,n}$ must have at least $n$ edges). Further if $|O\setminus F|=n$ then $|F|\leq 2n-|O\setminus F|=2n-n=n$ and since $|F|\geq n$ we have $|F|=n$ in this case.
\end{proof}

By Theorem \ref{thm:generalTP} we already know that this bound is tight for all $n$. We close our discussion with an explicit example with $\CD_{fm}=n-1$ for all $n$:

\begin{example} \label{ex:TP diameters coincide}
Consider the (non-degenerate) transportation polytope given by margins $u_1=u_2=2n-1$, $v_1=2n$, $v_j=2$ for $j=2,\ldots,n$. The two assignments below have circuit distance $n-1$ as every circuit step can insert at most one edge incident to $\supply_1$ and we have to add $n-1$ such edges.
	
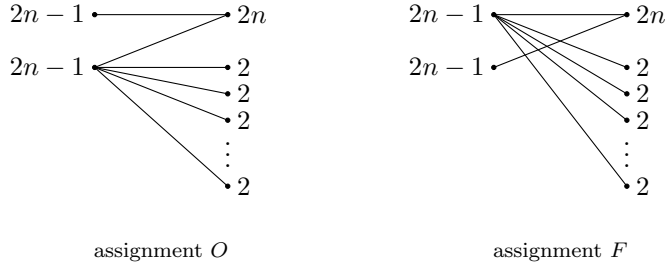
\begin{figure}[H] 
		\centering
			\begin{tikzpicture}[scale=0.35]
					\node at (2.5,-2.5) {\footnotesize assignment $O$};
					\coordinate (c1) at (0,6.5);
					\coordinate (c2) at (0,4.5);
					\coordinate (x1) at (5,6.5);
					\coordinate (x2) at (5,4.5);
					\coordinate (x3) at (5,3.5);
					\coordinate (x4) at (5,2.5);
					\coordinate (dots) at (5,1.5);
					\coordinate (x5) at (5,0);				
					\draw [fill, black] (c1) circle (0.08cm);
					\draw [fill, black] (c2) circle (0.08cm);
					\draw [fill, black] (x2) circle (0.08cm);					
				  \draw [fill, black] (x3) circle (0.08cm);					
				  \draw [fill, black] (x4) circle (0.08cm);					
				  \draw [fill, black] (x5) circle (0.08cm);					
				  \draw [fill, black] (x1) circle (0.08cm);					
				  \node[left] at (c1) {$2n-1$};
				  \node[left] at (c2) {$2n-1$};
				  \node[right] at (x1) {$2n$};
				  \node[right] at (x2) {$2$};
				  \node[right] at (x3) {$2$};
				  \node[right] at (x4) {$2$};
				  \node[right] at (x5) {$2$};
				  \node at (dots) {\vdots};
				  \draw (c1)--(x1);
				  \draw (c2)--(x1);
				  \draw (c2)--(x2);
				  \draw (c2)--(x3);
				  \draw (c2)--(x4);
				  \draw (c2)--(x5);
			\end{tikzpicture}
		\hspace{4em}
			\begin{tikzpicture}[scale=0.35]
					\node at (2.5,-2.5) {\footnotesize assignment $F$};
					\coordinate (c1) at (0,6.5);
					\coordinate (c2) at (0,4.5);
					\coordinate (x1) at (5,6.5);
					\coordinate (x2) at (5,4.5);
					\coordinate (x3) at (5,3.5);
					\coordinate (x4) at (5,2.5);
					\coordinate (dots) at (5,1.5);
					\coordinate (x5) at (5,0);				
					\draw [fill, black] (c1) circle (0.08cm);
					\draw [fill, black] (c2) circle (0.08cm);
					\draw [fill, black] (x2) circle (0.08cm);					
				  \draw [fill, black] (x3) circle (0.08cm);					
				  \draw [fill, black] (x4) circle (0.08cm);					
				  \draw [fill, black] (x5) circle (0.08cm);					
				  \draw [fill, black] (x1) circle (0.08cm);					
				  \node[left] at (c1) {$2n-1$};
				  \node[left] at (c2) {$2n-1$};
				  \node[right] at (x1) {$2n$};
				  \node[right] at (x2) {$2$};
				  \node[right] at (x3) {$2$};
				  \node[right] at (x4) {$2$};
				  \node[right] at (x5) {$2$};
				  \node at (dots) {\vdots};
				  \draw (c1)--(x1);
				  \draw (c2)--(x1);
				  \draw (c1)--(x2);
				  \draw (c1)--(x3);
				  \draw (c1)--(x4);
				  \draw (c1)--(x5);
			\end{tikzpicture}
	\end{figure}				
Hence this transportation polytope has circuit diameter equal to $n-1$. We also have graph diameter $n-1$, as $\{\supply_1,\demand_1\}$ and $\{\supply_2,\demand_1\}$ are critical edges.
\end{example}


By a similar analysis, one can also prove that for vertices $O$ and $F$ of a $3{\times}n$--transportation polytope the circuit distance $\CD_{fm}$ from $O$ to $F$ is at most $|O\backslash F|$. The analysis becomes more involved, as the circuits are not characterized as easily as in the $2{\times}n$ case anymore.
\vspace*{1cm}
\section{The graph diameter for $2{\times}n$-- and $3{\times}n$--transportation polytopes}

In this section, we prove the monotone Hirsch conjecture with a bound of $n$ for $2{\times}n$--transportation polytopes and the Hirsch conjecture for $3{\times}n$--transportation polytopes. We split the proofs into small parts, beginning with a marking system that is at the core of our approach.

\subsection{A marking system for the graph diameter}

The key parts of our proofs of the graph diameters for $2{\times}n$-- and $3{\times}n$--transportation polytopes will be based on a marking system. During the walk from an assignment $O$ to an assignment $F$, we distinguish marked and unmarked edges for the current assignment. 
\begin{itemize}
    \item \textbf{Unmarked edges}: May be deleted.
    \item \textbf{Marked edges}: Must not be deleted. For every marked edge $\{\supply_i,\demand_j\}$, either
 \begin{enumerate}
            \item $\demand_j$ is a leaf demand in $F$, or
            \item $\demand_j$ is a mixed demand in $F$ and all leaf edges (in $F$) incident to $\supply_i$ already exist in $O$ and are marked.
        \end{enumerate}
\end{itemize}
The general idea is that after at most one pivot we may always mark an edge in our current assignment. Throughout the whole process we do not delete any marked edges. Thus, we will need at most $|F|$ steps to get to the final assignment. This approach proves an upper bound of $|F|$ on the combinatorial diameter of $2{\times}n$-- and $3{\times}n$--transportation polytopes. By refining these arguments we show the Hirsch conjecture for $m=2,3$ and in the case of $2{\times}n$--transportation polytopes improve the diameter bound by one to $n$.

Before starting with the proofs, let us outline some general conventions, situations, and arguments that frequently appear in our analysis.

In the sketches throughout this section, marked edges are drawn in bold, while unmarked edges are drawn plainly. Edges that are possibly marked are depicted as a plain line with a dashed bold line over it. 

We construct edge-walks, so we do not have to distinguish between assignments and assignments that are vertices. For convenience, we continue to say `assignments'. When talking about `mixed' or `leaf' edges without refering to a specific assignment, we always mean the edges are mixed or leaves in the final assignment $F$. For example, `$\supply_i$ has all its leaf edges' means that all leaf edges in $F$ that are incident to $\supply_i$ in $F$ also exist in $O$.

Recall that $\{\supply_i,\demand_j\}$ is an edge to increase if $y^O_{ij}<y^F_{ij}$ and an edge to decrease if $y^O_{ij}>y^F_{ij}$. Clearly, if there is an edge to increase incident to a node $\supply_i$ or $\demand_j$ in $O$, there also must be an edge incident to this node to decrease and vice versa. Further observe that edges to increase must be edges to insert or edges that are mixed in $O$, while edges to decrease are edges we have to delete or edges that exist in $O$ and are mixed in $F$. These principles are frequently used in our proofs.

Note also if there exist marked edges $\{\supply_{i_1},\demand_j\}, \{\supply_{i_2},\demand_j\}\in{O}$, then $\supply_{i_1}$ and $\supply_{i_2}$ already have all their leaf demands, as $\{\supply_{i_1},\demand_j\}$ and $\{\supply_{i_2},\demand_j\}$ are marked and thus in $F$. In particular these edges are mixed in $F$. But these mixed edges can only be marked if all leaf edges incident to $\supply_{i_1}$ and $\supply_{i_2}$ already exist in $O$. In particular this implies that $O=F$ if all edges that are mixed in $O$ are marked.


Before turning to the $2{\times}n$-- and $3{\times}n$ cases, we present a lemma which will be useful for completing our proofs in many configurations. We state it only for these cases, but it is readily extendable to general $m{\times}n$--transportation polytopes.

In this lemma, when we describe an assignment as `marked', we assume it was obtained by the rules described above. By saying that `an edge $e\in E_m^O$ is an even number of edges away from $\supply_i$ in $E_m^O$', we state that the path with edges in $E_m^O$ ending with $e$ and starting at node $\supply_i$ has an even number of edges.

\begin{lemma}\label{marklem} Let $O\neq F$ be two assignments in a non-degenerate $2{\times}n$-- or $3{\times}n$--transportation polytope, with $O$ partially marked. Then if there exists some $\supply_i$, such that all marked edges in $E_m^O$ are an even number of edges away from $\supply_i$ in $E_m^O$, then after (at most) one pivot, we may mark some edge in $O$. 

\end{lemma}

\begin{proof} The condition that all marked edges in $E_m^O$ are an even number of edges away from $\supply_i$ in $E_m^O$ implies that inserting any edge incident to $\supply_i$ will never delete a marked edge in $E_m^O$, since they will always be increasing on that pivot. Thus, when inserting any edge, we only have to worry about deleting marked edges not in $E_m^O$. We proceed then as follows (starting from the top):

\begin{enumerate}
    \item If there is an unmarked leaf edge $\{\supply_i,\demand_j\}$, we insert it (if necessary) and mark it. Since this is a leaf edge in $F$, the decreasing edge incident to $\demand_j$ is not marked.

    \item Else if $\{\supply_i,\demand_j\}\in E_m^F\cap O$, since $\supply_i$ has all its leaf edges, we mark it.

    \item Else if there is only one mixed edge $\{\supply_i,\demand_j\}\in F$ to insert, this implies $\supply_i$ has all its leaf edges and is incident to only one mixed edge in $F$ (else its other mixed edge would already have been inserted and either we would have applied the step $(2)$, or it would have already been marked and the lemma could not have been applied). We insert and mark the edge $\{\supply_i,\demand_j\}$.

If this pivot deleted a marked edge $\{\supply_k,\demand_j\}$ (the edge incident to $\demand_j$ in $O$ and in our pivot), then $\demand_j$ is a leaf in $O$ (otherwise there would be a marked edge an even number of edges away from $\supply_i$ in $E_m^O$). Then $\supply_i$ would still be incident to an edge to decrease ($\{\supply_i,\demand_j\}$, since it is now a leaf, but mixed in $F$), but would be incident to no edge to increase, since it has all its leaves and has only one mixed edge in $F$. Hence no marked edge will be deleted.

\item Else, since by step (2), $\supply_i$ is incident to no edges in $O\cap E_m^F$, and is incident to some unmarked edge in $E_m^O$ by assumption, it must have at least one edge to delete and one to insert. Since $\supply_i$ has all its leaf edges and more than one mixed edge to insert, by steps (1) and (3), we know that $\supply_i$ has two mixed edges incident to it to insert, $\{\supply_i,\demand_j\}$ for $j=1,2$ (Note this does not happen in the $2{\times}n$ case). 

If for some $j$, $\demand_j$ is not incident to a marked edge, we insert $\{\supply_i,\demand_j\}$ and mark it. This will not delete a marked edge

	Else, both $\demand_j$ are incident to marked edges, which must necessarily be mixed in $F$ since $\demand_j$ are mixed in $F$. Without loss of generality let these edges be $\{\supply_j,\demand_j\}$ for $j=1,2$ and let $\supply_i=\supply_3$. Note this implies both $\supply_j$ already have all their leaf edges marked, since they are incident to a marked mixed edge. Then for some $j$, we must have the following configuration in $O$ with $\{\supply_j,\demand_k\},\{\supply_3,\demand_k\}\in E_m^O$.
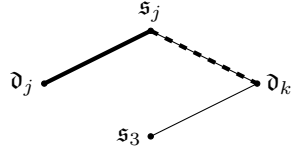
 \begin{figure}[H]
        \centering
            \begin{tikzpicture}[scale=0.35]
                    \coordinate (cj) at (0,4);
                    \coordinate (xk) at (4,2);
                    \coordinate (c3) at (0,0);
                    \coordinate (xj) at (-4,2);
                    \draw [fill, black] (cj) circle (0.1cm);
                    \draw [fill, black] (xk) circle (0.1cm);
                    \draw [fill, black] (c3) circle (0.1cm);
                    \draw [fill, black] (xj) circle (0.1cm);
                    \node[above] at (cj) {$\supply_j$};
                    \node[left] at (c3) {$\supply_3$};
                    \node[left] at (xj) {$\demand_j$};
                    \node[right] at (xk) {$\demand_k$};
                    \draw[ultra thick, dashed](cj)--(xk);
		  \draw (cj)--(xk);
                    \draw(c3) --(xk);
                    \draw[ultra thick](cj) --(xj);
            \end{tikzpicture}
        \end{figure}
 The existence of such a path is trivial to see, since $O$ is connected. The marking patterns of this configuration will be one of the following:
 
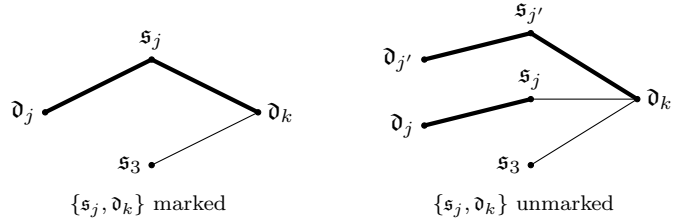
\begin{figure}[H]
        \centering
	\subfloat[$\{\supply_j,\demand_k\}$ marked]{
            \begin{tikzpicture}[scale=0.35]
                    \coordinate (cj) at (0,4);
                    \coordinate (xk) at (4,2);
                    \coordinate (c3) at (0,0);
                    \coordinate (xj) at (-4,2);
                    \draw [fill, black] (cj) circle (0.1cm);
                    \draw [fill, black] (xk) circle (0.1cm);
                    \draw [fill, black] (c3) circle (0.1cm);
                    \draw [fill, black] (xj) circle (0.1cm);
                    \node[above] at (cj) {$\supply_j$};
                    \node[left] at (c3) {$\supply_3$};
                    \node[left] at (xj) {$\demand_j$};
                    \node[right] at (xk) {$\demand_k$};
                    \draw[ultra thick](cj)--(xk);
                    \draw(c3) --(xk);
                    \draw[ultra thick](cj) --(xj);
            \end{tikzpicture}}
\hspace{2em}
	\subfloat[$\{\supply_j,\demand_k\}$ unmarked]{
            \begin{tikzpicture}[scale=0.35]
                    \coordinate (cj) at (0,2.5);
                    \coordinate (xk) at (4,2.5);
                    \coordinate (c3) at (0,0);
                    \coordinate (xj) at (-4,1.5);
		\coordinate (cjj) at (0,5);
		\coordinate (xjj) at (-4,4);
                    \draw [fill, black] (cj) circle (0.1cm);
                    \draw [fill, black] (xk) circle (0.1cm);
                    \draw [fill, black] (c3) circle (0.1cm);
                    \draw [fill, black] (xj) circle (0.1cm);
                    \draw [fill, black] (xjj) circle (0.1cm);
                    \draw [fill, black] (cjj) circle (0.1cm);
                    \node[above] at (cj) {$\supply_j$};
                    \node[left] at (c3) {$\supply_3$};
                    \node[left] at (xj) {$\demand_j$};
                    \node[right] at (xk) {$\demand_k$};
                    \node[left] at (xjj) {$\demand_{j'}$};
                    \node[above] at (cjj) {$\supply_{j'}$};
		  \draw (cj)--(xk);
                    \draw(c3) --(xk);
                    \draw[ultra thick](cj) --(xj);
		\draw[ultra thick](xk)--(cjj);
		\draw[ultra thick](cjj)--(xjj);

            \end{tikzpicture}}
        \end{figure}

The fact that these are the only two possibilities follows since all our leaf edges are already inserted and marked. If $\{\supply_j,\demand_k\}$ is not marked, then $\demand_k$ must be a leaf demand in $F$ and hence $\{\supply_{j'},\demand_k\}$ must be a leaf edge in $F$ and therefore must already be in $O$ and marked. Because these are the only possibilities, we may assume without loss of generality that $\{\supply_j,\demand_k\}$ is marked (otherwise we could have swapped the roles of $j$ and $j'$ when choosing our path).


If $\supply_j$ is incident to no other edges in $E_m^O$ besides $\{\supply_j,\demand_k\}$, inserting $\{\supply_3,\demand_j\}$ will not delete $\{\supply_j,\demand_j\}$ as otherwise $\supply_j$ would have edges to increase ($\{\supply_j,\demand_j\}$ since it is mixed in $F$), but no edges to decrease (as all leaf edges incident to $\supply_j$ are marked). Hence we insert and mark $\{\supply_3,\demand_j\}$.

Else, $\supply_j$ is incident to one more edges in $E_m^O$ (which cannot be marked since it is an odd number of edges away from $\supply_3$), so $O$ looks like:
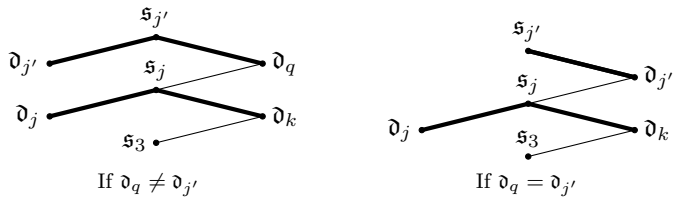
\begin{figure}[H]
        \centering
	\subfloat[If $\demand_q\neq \demand_{j'}$]{
            \begin{tikzpicture}[scale=0.35]
		\coordinate (cjj) at (0,4);
		\coordinate (xjj) at (-4,3);
		\coordinate (xq) at (4,3);
                    \coordinate (cj) at (0,2);
                    \coordinate (xk) at (4,1);
                    \coordinate (c3) at (0,0);
                    \coordinate (xj) at (-4,1);
                    \draw [fill, black] (cj) circle (0.1cm);
                    \draw [fill, black] (xk) circle (0.1cm);
                    \draw [fill, black] (c3) circle (0.1cm);
                    \draw [fill, black] (xj) circle (0.1cm);
                    \draw [fill, black] (cjj) circle (0.1cm);
                    \draw [fill, black] (xjj) circle (0.1cm);
                    \draw [fill, black] (xq) circle (0.1cm);
                    \node[above] at (cj) {$\supply_j$};
                    \node[left] at (c3) {$\supply_3$};
                    \node[left] at (xj) {$\demand_j$};
                    \node[right] at (xk) {$\demand_k$};
                    \node[above] at (cjj) {$\supply_{j'}$};
                    \node[left] at (xjj) {$\demand_{j'}$};
                    \node[right] at (xq) {$\demand_q$};
                    \draw[ultra thick](cj)--(xk);
                    \draw(c3) --(xk);
                    \draw[ultra thick](cj) --(xj);
                    \draw[ultra thick](cjj)--(xjj);
                    \draw[ultra thick](cjj)--(xq);
                    \draw(cj)--(xq);
            \end{tikzpicture}}
\hspace{2em}
\subfloat[If $\demand_q=\demand_{j'}$]{
\begin{tikzpicture}[scale=0.35]
		\coordinate (cjj) at (0,4);
		\coordinate (xjj) at (4,3);
                    \coordinate (cj) at (0,2);
                    \coordinate (xk) at (4,1);
                    \coordinate (c3) at (0,0);
                    \coordinate (xj) at (-4,1);
                    \draw [fill, black] (cj) circle (0.1cm);
                    \draw [fill, black] (xk) circle (0.1cm);
                    \draw [fill, black] (c3) circle (0.1cm);
                    \draw [fill, black] (xj) circle (0.1cm);
                    \draw [fill, black] (cjj) circle (0.1cm);
                    \draw [fill, black] (xjj) circle (0.1cm);
                    \node[above] at (cj) {$\supply_j$};
                    \node[above] at (c3) {$\supply_3$};
                    \node[left] at (xj) {$\demand_j$};
                    \node[right] at (xk) {$\demand_k$};
                    \node[above] at (cjj) {$\supply_{j'}$};
                    \node[right] at (xjj) {$\demand_{j'}$};
                    \draw[ultra thick](cj)--(xk);
                    \draw(c3) --(xk);
                    \draw[ultra thick](cj) --(xj);
                    \draw[ultra thick](cjj)--(xjj);
                    \draw[ultra thick](cjj)--(xq);
                    \draw(cj)--(xq);
            \end{tikzpicture}}
        \end{figure}
 In the first case we insert and mark the edge $\{\supply_3,\demand_{j'}\}$, and here we again know that $\{\supply_{j'},\demand_{j'}\}$ will not be deleted since otherwise $\supply_{j'}$ would have an edge to increase, but none to decrease, since $\demand_q$ must be a leaf. In the second case we insert and mark the edge $\{\supply_3,\demand_{j'}\}$, and clearly no marked edges will be deleted.
\end{enumerate}

This proves the claim.
\end{proof}

The proof of Lemma \ref{marklem} essentially is an algorithm to decide which pivot to use. We refer to using this algorithm as \emph{applying Lemma \ref{marklem}}.
\subsection{$2 {\times}n$--transportation polytopes}
\noindent 
 Validity of the Hirsch conjecture for a $2{\times}n$--transportation polytope implies a general upper bound of $n+1$ on its graph diameter (or $n+1-k$, where $k$ is the number of critical edges in the polytope). In this section, we prove that the diameter of $2{\times}n$--transportation polytopes is actually bounded by $n$ and is tight in the sense that there is a $2{\times}n$--transporation polytope that has diameter $n$ for all $n\geq 3$. In the next section we then show that the corresponding edge walk satisfies the monotone Hirsch conjecture. 


Let $O \neq F$ be two assignments of a non-degenerate $2{\times}n$--transportation polytope. (Recall that it suffices to consider non-degenerate polytopes for the graph diameter.) Lemma \ref{lem0} also tells us that $O$ and $F$ differ by at most $n-1$ edges in this non-degenerate case. In fact, there are transportation polytopes with assignments $O$ and $F$ such that $|O\backslash F|=n-1$, as exhibited by Example \ref{ex:TP diameters coincide}.

Recall Example \ref{THEex1} from Section $2$. It illustrates a situation in which we cannot decrease the edge distance with every pivot, such that we have $\CD^O_e=3>2=|O\backslash F|$. In particular, in the first step, no matter which edge in $F\backslash O$ we insert (dashed edges), we have to delete an edge that is contained in $F$. Observe that the edge deleted is a mixed edge in both $O$ and $F$ and we have $D_m^O= D_m^F$. In our proof of the Hirsch conjecture for $2{\times}n$--transportation polytopes this will be the situation we have to take special care of. 

	
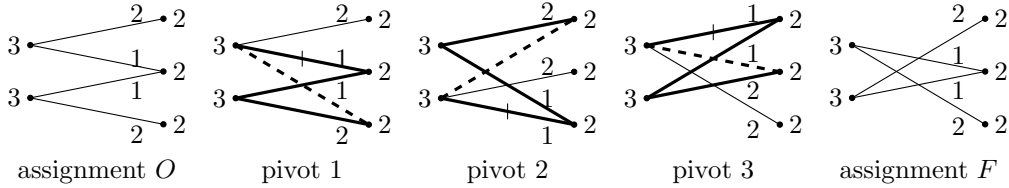
\begin{figure}[H]
		\centering
			\begin{tikzpicture}[scale=0.35]
					\coordinate (c1) at (0,4);
					\coordinate (c2) at (0,2);
					\coordinate (x1) at	(5,5);	
					\coordinate (x2) at (5,3);							
					\coordinate (x3) at (5,1);
					\draw [fill, black] (c1) circle (0.1cm);
					\draw [fill, black] (c2) circle (0.1cm);
					\draw [fill, black] (x1) circle (0.1cm);
					\draw [fill, black] (x2) circle (0.1cm);
					\draw [fill, black] (x3) circle (0.1cm);
					\node[left] at (c1) {$3$};
					\node[left] at (c2) {$3$};
					\node[right] at (x1) {$2$};
					\node[right] at (x2) {$2$};
				  \node[right] at (x3) {$2$};
				  \draw (c1)--(x2);
				  \draw (c2)--(x2);
					\draw (c1)--(x1);
				  \draw (c2)--(x3);			  
				  \node at (4,5.2) {$2$};
					\node at (4,3.5) {$1$};
					\node at (4,2.3) {$1$};
					\node at (4,0.6) {$2$};					
					\node at (2.5,-0.8) {assignment $O$};			
			\end{tikzpicture}
			\begin{tikzpicture}[scale=0.35]
					\coordinate (c1) at (0,4);
					\coordinate (c2) at (0,2);
					\coordinate (x1) at	(5,5);	
					\coordinate (x2) at (5,3);							
					\coordinate (x3) at (5,1);
					\draw [fill, black] (c1) circle (0.1cm);
					\draw [fill, black] (c2) circle (0.1cm);
					\draw [fill, black] (x1) circle (0.1cm);
					\draw [fill, black] (x2) circle (0.1cm);
					\draw [fill, black] (x3) circle (0.1cm);
					\node[left] at (c1) {$3$};
					\node[left] at (c2) {$3$};
					\node[right] at (x1) {$2$};
					\node[right] at (x2) {$2$};
				  \node[right] at (x3) {$2$};
				  \draw[very thick]  (c1)--(x2);
				  \draw[very thick]  (c2)--(x2);
				  \draw (2.5,3.8)--(2.5,3.2);
					\draw (c1)--(x1);
				  \draw[very thick]  (c2)--(x3);
				  \draw[dashed, very thick] (c1)--(x3);				  
				  \node at (4,5.2) {$2$};
					\node at (4,3.5) {$1$};
					\node at (4,2.3) {$1$};
					\node at (4,0.6) {$2$};	
					\node at (2.5,-0.8) {pivot $1$};										
			\end{tikzpicture}
			\begin{tikzpicture}[scale=0.35]
					\coordinate (c1) at (0,4);
					\coordinate (c2) at (0,2);
					\coordinate (x1) at	(5,5);	
					\coordinate (x2) at (5,3);							
					\coordinate (x3) at (5,1);
					\draw [fill, black] (c1) circle (0.1cm);
					\draw [fill, black] (c2) circle (0.1cm);
					\draw [fill, black] (x1) circle (0.1cm);
					\draw [fill, black] (x2) circle (0.1cm);
					\draw [fill, black] (x3) circle (0.1cm);
					\node[left] at (c1) {$3$};
					\node[left] at (c2) {$3$};
					\node[right] at (x1) {$2$};
					\node[right] at (x2) {$2$};
				  \node[right] at (x3) {$2$};
				  \draw[very thick]  (c1)--(x1);
				  \draw[dashed, very thick] (c2)--(x1);
				  \draw[very thick]  (c1)--(x3);
				  \draw (c2)--(x2);
					\draw[very thick]  (c2)--(x3);
					\draw (2.5,1.8)--(2.5,1.2);					
					\node at (4,5.3) {$2$};
					\node at (4,3.2) {$2$};
					\node at (4,2.0) {$1$};
					\node at (4,0.6) {$1$};
					\node at (2.5,-0.8) {pivot $2$};										
			\end{tikzpicture}
			\begin{tikzpicture}[scale=0.35]
					\coordinate (c1) at (0,4);
					\coordinate (c2) at (0,2);
					\coordinate (x1) at	(5,5);	
					\coordinate (x2) at (5,3);							
					\coordinate (x3) at (5,1);
					\draw [fill, black] (c1) circle (0.1cm);
					\draw [fill, black] (c2) circle (0.1cm);
					\draw [fill, black] (x1) circle (0.1cm);
					\draw [fill, black] (x2) circle (0.1cm);
					\draw [fill, black] (x3) circle (0.1cm);
					\node[left] at (c1) {$3$};
					\node[left] at (c2) {$3$};
					\node[right] at (x1) {$2$};
					\node[right] at (x2) {$2$};
				  \node[right] at (x3) {$2$};
				  \draw[very thick]  (c1)--(x1);
				  \draw (c1)--(x3);
				  \draw[very thick]  (c2)--(x1);
					\draw[very thick]  (c2)--(x2);
					\draw[dashed, very thick] (c1)--(x2);
					\draw (2.5,4.8)--(2.5,4.2);
					\node at (4,5.2) {$1$};
					\node at (4,3.7) {$1$};
					\node at (4,2.3) {$2$};
					\node at (4,0.8) {$2$};
					\node at (2.5,-0.8) {pivot $3$};											
			\end{tikzpicture}
			\begin{tikzpicture}[scale=0.35]
					\coordinate (c1) at (0,4);
					\coordinate (c2) at (0,2);
					\coordinate (x1) at	(5,5);	
					\coordinate (x2) at (5,3);							
					\coordinate (x3) at (5,1);
					\draw [fill, black] (c1) circle (0.1cm);
					\draw [fill, black] (c2) circle (0.1cm);
					\draw [fill, black] (x1) circle (0.1cm);
					\draw [fill, black] (x2) circle (0.1cm);
					\draw [fill, black] (x3) circle (0.1cm);
					\node[left] at (c1) {$3$};
					\node[left] at (c2) {$3$};
					\node[right] at (x1) {$2$};
					\node[right] at (x2) {$2$};
				  \node[right] at (x3) {$2$};
				  \draw (c1)--(x3);
				  \draw (c1)--(x2);
				  \draw (c2)--(x1);
					\draw (c2)--(x2);					
					\node at (4,5.2) {$2$};
					\node at (4,3.6) {$1$};
					\node at (4,2.3) {$1$};
					\node at (4,0.9) {$2$};			
					\node at (2.5,-0.8) {assignment $F$};											
			\end{tikzpicture}
			\end{figure} 

A similar example can be constructed to see that there are $2{\times}n$--transportation polytopes with diameter at least $n$ for all $n\geq 3$.

\begin{example}
Consider an instance of a transportation problem with margins $u_1=u_2= 2n-3$, $v_1=2n-4$, and $v_j=2$ for all $j=2,\ldots, n$. This yields a non-degenerate transportation polytope. Consider the following assignments $O$ and $F$. 
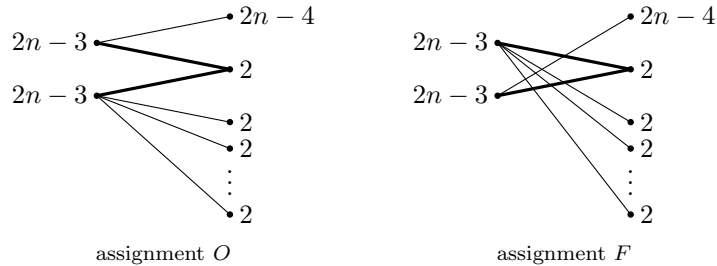
	\begin{figure}[H] 
		\centering
		\subfloat[assignment $O$]{
			\begin{tikzpicture}[scale=0.35]
					\coordinate (c1) at (0,6.5);
					\coordinate (c2) at (0,4.5);
					\coordinate (x1) at (5,7.5);
					\coordinate (x2) at (5,5.5);
					\coordinate (x3) at (5,3.5);
					\coordinate (x4) at (5,2.5);
					\coordinate (dots) at (5,1.5);
					\coordinate (x5) at (5,0);				
					\draw [fill, black] (c1) circle (0.1cm);
					\draw [fill, black] (c2) circle (0.1cm);
					\draw [fill, black] (x2) circle (0.1cm);					
				  \draw [fill, black] (x3) circle (0.1cm);					
				  \draw [fill, black] (x4) circle (0.1cm);					
				  \draw [fill, black] (x5) circle (0.1cm);					
				  \draw [fill, black] (x1) circle (0.1cm);					
				  \node[left] at (c1) {$2n-3$};
				  \node[left] at (c2) {$2n-3$};
				  \node[right] at (x1) {$2n-4$};
				  \node[right] at (x2) {$2$};
				  \node[right] at (x3) {$2$};
				  \node[right] at (x4) {$2$};
				  \node[right] at (x5) {$2$};
				  \node at (dots) {\vdots};
				  \draw[very thick] (c1)--(x2);
				  \draw[very thick] (c2)--(x2);
				  \draw (c1)--(x1);
				  \draw (c2)--(x3);
				  \draw (c2)--(x4);
				  \draw (c2)--(x5);
			\end{tikzpicture}
		}
		\hspace{2em}
		\subfloat[assignment $F$]{
			\begin{tikzpicture}[scale=0.35]
					\coordinate (c1) at (0,6.5);
					\coordinate (c2) at (0,4.5);
					\coordinate (x1) at (5,7.5);
					\coordinate (x2) at (5,5.5);
					\coordinate (x3) at (5,3.5);
					\coordinate (x4) at (5,2.5);
					\coordinate (dots) at (5,1.5);
					\coordinate (x5) at (5,0);				
					\draw [fill, black] (c1) circle (0.1cm);
					\draw [fill, black] (c2) circle (0.1cm);
					\draw [fill, black] (x2) circle (0.1cm);					
				  \draw [fill, black] (x3) circle (0.1cm);					
				  \draw [fill, black] (x4) circle (0.1cm);					
				  \draw [fill, black] (x5) circle (0.1cm);					
				  \draw [fill, black] (x1) circle (0.1cm);					
				  \node[left] at (c1) {$2n-3$};
				  \node[left] at (c2) {$2n-3$};
				  \node[right] at (x1) {$2n-4$};
				  \node[right] at (x2) {$2$};
				  \node[right] at (x3) {$2$};
				  \node[right] at (x4) {$2$};
				  \node[right] at (x5) {$2$};
				  \node at (dots) {\vdots};
				  \draw[very thick] (c1)--(x2);
				  \draw[very thick] (c2)--(x2);
				  \draw (c2)--(x1);
				  \draw (c1)--(x3);
				  \draw (c1)--(x4);
				  \draw (c1)--(x5);
			\end{tikzpicture}
		}			
	\end{figure}				
Inserting any edge into $O$ creates a new assignment $C$ with $|C\backslash F|=n-1$ since the edge deleted will always be in $O\cap F$. Since $\CD_e^O=\CD_e^{C}+1$ for some $C$ we have that $\CD_e^O=\CD_e^{C}+1\geq |C\backslash F| + 1= (n-1) + 1 = n.$
\end{example}

The above example tells us that the upper bound of $n$ we prove in the following is tight for all $n\geq 3$ in the sense that there are margins such that the diameter is $n$. We first need a lemma about our marking system in the $2{\times}n$ case before proving the upper bound of $n$.

\begin{lemma}\label{2nmark} Let $O\neq F$ be two assignments in a non-degenerate $2{\times}n$-transportation polytope and let $O$ be partially marked. Suppose there is some edge $e\in O\cap E_m^F$ such that either $e$ is a leaf edge in $O$, or $E_m^O=E_m^F$ and the other mixed edge in $O$ is marked. Then applying Lemma \ref{marklem} to any $\supply_i$ will not delete $e$.
\end{lemma}

\begin{proof} Let $e=\{\supply_2,\demand_1\}$. If $e$ is a leaf edge in $O$, then if Lemma \ref{marklem} inserts a leaf edge in $F$, clearly $e$ will not be deleted as it is not part of the pivot. Otherwise we either have $E_m^O=E_m^F$ with $\{\supply_1,\demand_1\}\in E_m^F$ marked, or we are inserting $\{\supply_1,\demand_1\}$ and $\{\supply_2,\demand_1\}$ is a leaf edge in $O$. Note in the case that $E_m^O=E_m^F$, we must be applying Lemma \ref{marklem} to $\supply_2$, since $\supply_1$ has a marked edge incident to it in the mixed part of $O$.
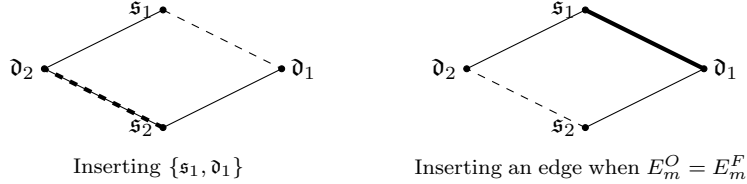
      \begin{figure}[H]
\centering
	\subfloat[Inserting $\{\supply_1,\demand_1\}$]{
           \begin{tikzpicture}[scale=0.39]
		\coordinate (c1) at (0,4);
		\coordinate (c2) at (0,0);
		\coordinate (x1) at (4,2);
		\coordinate (x2) at (-4,2);
		\draw[fill, black] (c1) circle (0.1cm);
		\draw[fill, black] (c2) circle (0.1cm);
		\draw[fill, black] (x1) circle (0.1cm);
		\draw[fill, black] (x2) circle (0.1cm);
		\node[left] at (c1) {$\supply_1$};
		\node[left] at (c2) {$\supply_2$};
		\node[right] at (x1) {$\demand_1$};
		\node[left] at (x2) {$\demand_2$};
		\draw[dashed] (c1)--(x1);
		\draw (c2)--(x1);
		\draw (c2)--(x2);
		\draw (c1)--(x2);
		\draw[dashed, ultra thick] (c2)--(x2);
	\end{tikzpicture}}
\hspace{3em}
	\subfloat[Inserting an edge when $E_m^O=E_m^F$]{
           \begin{tikzpicture}[scale=0.39]
		\coordinate (c1) at (0,4); 
		\coordinate (c2) at (0,0);
		\coordinate (x1) at (4,2);
		\coordinate (x2) at (-4,2);
		\draw[fill, black] (c1) circle (0.1cm);
		\draw[fill, black] (c2) circle (0.1cm);
		\draw[fill, black] (x1) circle (0.1cm);
		\draw[fill, black] (x2) circle (0.1cm);
		\node[left] at (c1) {$\supply_1$};
		\node[left] at (c2) {$\supply_2$};
		\node[right] at (x1) {$\demand_1$};
		\node[left] at (x2) {$\demand_2$};
		\draw[ultra thick] (c1)--(x1);
		\draw (c2)--(x1);
		\draw (c1)--(x2);
		\draw[dashed] (c2)--(x2);
	\end{tikzpicture}}
        \end{figure}
Note in both configurations we know that $\supply_1$ must have all its leaf edges (since it either has a marked mixed edge or one is being inserted). Therefore if Lemma \ref{marklem} has us inserting some edge (all such cases shown above), $\{\supply_2,\demand_1\}$ will not be deleted. Otherwise, $\supply_1$ would have an edge to decrease ($\{\supply_1,\demand_1\}$) but none to increase.
\end{proof}

\begin{lemma}\label{thm2n}
The diameter of a $2{\times}n$--transportation polytope is at most $n$.
\end{lemma}

\begin{proof} Starting with a partially marked assignment $O$, we will show that after (at most) one pivot we may obtain a new assignment (which we will also call $O$) where we can mark one edge in $F$. If $O$ is partially marked, its set of mixed edges looks like one of the following cases:
      
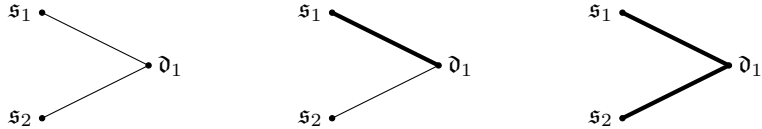
\begin{figure}[H]
\centering
	\subfloat{
           \begin{tikzpicture}[scale=0.35]
		\coordinate (c1) at (0,4);
		\coordinate (c2) at (0,0);
		\coordinate (x1) at (4,2);
		\draw[fill, black] (c1) circle (0.1cm);
		\draw[fill, black] (c2) circle (0.1cm);
		\draw[fill, black] (x1) circle (0.1cm);
		\node[left] at (c1) {$\supply_1$};
		\node[left] at (c2) {$\supply_2$};
		\node[right] at (x1) {$\demand_1$};
		\draw (c1)--(x1);
		\draw (c2)--(x1);
	\end{tikzpicture}}
\hspace{3em}
	\subfloat{
           \begin{tikzpicture}[scale=0.35]
		\coordinate (c1) at (0,4);
		\coordinate (c2) at (0,0);
		\coordinate (x1) at (4,2);
		\draw[fill, black] (c1) circle (0.1cm);
		\draw[fill, black] (c2) circle (0.1cm);
		\draw[fill, black] (x1) circle (0.1cm);
		\node[left] at (c1) {$\supply_1$};
		\node[left] at (c2) {$\supply_2$};
		\node[right] at (x1) {$\demand_1$};
		\draw[ultra thick] (c1)--(x1);
		\draw (c2)--(x1);
	\end{tikzpicture}}
\hspace{3em}
	\subfloat{
           \begin{tikzpicture}[scale=0.35]
		\coordinate (c1) at (0,4);
		\coordinate (c2) at (0,0);
		\coordinate (x1) at (4,2);
		\draw[fill, black] (c1) circle (0.1cm);
		\draw[fill, black] (c2) circle (0.1cm);
		\draw[fill, black] (x1) circle (0.1cm);
		\node[left] at (c1) {$\supply_1$};
		\node[left] at (c2) {$\supply_2$};
		\node[right] at (x1) {$\demand_1$};
		\draw[ultra thick] (c1)--(x1);
		\draw[ultra thick] (c2)--(x1);
	\end{tikzpicture}}
        \end{figure}
In the first two cases, we may apply Lemma \ref{marklem} to $\supply_2$ to obtain a marking after (at most) one pivot. In the third case, since both these edges are marked they must be mixed in $F$ and hence $E_m^O=E_m^F$ and since these mixed edges are marked we know that $\supply_1$ and $\supply_2$ both have all their leaf edges in $F$ and hence $O=F$.

As a consequence, this implies the diameter is at most $n+1$, since there are $n+1$ edges to mark in $F$ and each takes at most one pivot. To show the diameter is in fact $n$, we only need to show there will always be one edge in $F$ where we do not have to apply a pivot before marking it.
Note if there is a leaf edge $e\in F$ such that $e\in O$ when we start our marking process, we can mark $e$ without a pivot and our proof is complete. Otherwise, no such edge exists. Since we know $|O\cap F|\geq 2$ ($|O\cap F|=|O|+|F|-|O\cup F|\geq 2(n+1)-2n=2$ since $|O\cup F|\leq 2n$), this implies that these assignments share at least two mixed edges in $F$ and hence $E_m^O=E_m^F$. 

Let $E_m^O=E_m^F=\{e_1,e_2\}$. If neither $e_1$ nor $e_2$ is ever deleted when we perform our pivots, then clearly we will be able to mark them both without performing a pivot and the diameter of the associated polytope will be at most $n-1$. If one of them is deleted at some time during our process, without loss of generality say $e_1$, then $e_2$ will become a leaf edge. Thus, by Lemma \ref{2nmark}, $e_2$ will never be deleted after this, as when we apply Lemma \ref{marklem} $e_2$ will either be a leaf edge or once it is a mixed edge again, $e_1$ will be marked. Hence we will eventually be able to mark $e_2$ in this case without making a pivot and the diameter of the associated polytope is at most $n$. Thus we complete the proof.
\end{proof}

The proof of Lemma \ref{thm2n} actually tells us that the diameter is bounded above by $\min\{n,n+1-k\}$, where $k$ is the number of critical edges in the polytope, as such edges can never be deleted. As a consequence, the polytope satisfies the Hirsch conjecture with an upper bound of $n$, as stated in Theorem \ref{thm:2n}.

\subsection{Validity of the monotone Hirsch conjecture}

For an $m{\times}n$--transportation polytope $P$, the monotone Hirsch conjecture may be stated as follows:

\emph{Given a cost vector $s=(s_{11},\dots,s_{n1},\dots,s_{mn})^T\in \mathbb{R}^{m\times n}$ and a vertex $y^O$ of $P$, there always is an edge walk of length at most $m+n-1$ from $y^O$ to a vertex $y^F=\arg\max\limits_{y\in P} s^Ty$ that visits vertices in a sequence of nondecreasing objective function values.}

 We refer to the diameter of a polytope with respect to such a nondecreasing sequence of objective function values as the \emph{monotone diameter}. We prove that the edge walk constructed in the previous section yields such a sequence.

For $2{\times}n$--transportation polytopes, the vector $s=(s_{11},\dots,s_{n1},\dots,s_{2n})^T\in \mathbb{R}^{2n}$ already tells us what $y^F$ looks like.

\begin{lemma}\label{lem:whatdoesylooklike}
Let $s\in \mathbb{R}^{2n}$ satisfy $s_{1i}-s_{2i}\geq s_{1(i+1)}-s_{2(i+1)}$ for all $i\leq n-1$ and let $j$ be a maximal index such that $\sum\limits_{i=1}^{j-1} v_i<u_1$. Then the assignment $y^F$ defined by
\begin{itemize}
\item  $y^F_{1i}=v_i$ for $i<j$, $y^F_{1j}=u_1-\sum\limits_{i=1}^{j-1} v_i$, and $y^F_{1i}=0$ for $i>j,$
\item $y^F_{2i}=0$ for $i<j$, $y^F_{2j}=v_j-y_{1j}$, and $y^F_{2i}=v_i$ for $i>j$,
\end{itemize}
 is an optimizer for $\max s^Ty$.
\end{lemma}

\begin{proof}
The index $j$ and the values $y^F_{1j}, y^F_{2j}$ are well-defined, as there is no index $j'$ with $\sum\limits_{i=1}^{j'} v_i =u_1$ due to non-degeneracy of the transportation polytope. $\demand_j$ is the unique mixed demand in the assignment.

It suffices to prove that the reduced costs of all pivots possible at $y^F$ are non-positive. Such a pivot corresponds to inserting an edge to a demand $\demand_i$ for $i\neq j$. Noting that all $\demand_i$ belong to $\supply_1$ for $i<j$ and belong to $\supply_2$ for $i> j$ shows that the reduced costs satisfy $$(s_{2i}-s_{1i}) + (s_{1j}-s_{2j}) \leq 0 \text{ due to } s_{1j}-s_{2j} \leq s_{1i}-s_{2i} \text{ for } i <j$$
and $$(s_{1i}-s_{2i}) + (s_{2j}-s_{1j}) \leq 0 \text{ due to } s_{1j}-s_{2j} \geq s_{1i}-s_{2i} \text{ for } i >j.$$

\end{proof}
Note that the condition $s_{1i}-s_{2i}\geq s_{1(i+1)}-s_{2(i+1)}$  for all $i\leq n-1$ is no restriction, as it can simply be achieved by reindexing the supply nodes $\supply_i$. 


We now prove that there is a sequence of nondecreasing objective function values corresponding to the construction in the previous section. We split this proof into two parts, one for $D_m^O=D_m^F$ (which implies $E_m^O=E_m^F$ for $2{\times}n$--transportation polytopes) and one for $D_m^O\neq D_m^F$. We begin with $D_m^O=D_m^F$.

\begin{lemma}\label{lem:fracnondecreasing}
Let $s\in \mathbb{R}^{2n}$, let $O\neq F$, let $y^F$ be maximal for $s^Ty$ and let finally $D_m^O=D_m^F$. Then any pivot inserting an edge in $F\backslash O$ is non-decreasing. 
\end{lemma}

\begin{proof}
 Let $D_m^O=D_m^F=\{\demand_j\}$ and let $s$ satisfy the prerequisites of Lemma \ref{lem:whatdoesylooklike}, i.e{.} without loss of generality $s\in \mathbb{R}^{2n}$ satisfies $s_{1i}-s_{2i}\geq s_{1(i+1)}-s_{2(i+1)}$ for all $i\leq n-1$. Inserting an edge in $F\backslash O$ then means  inserting an edge $\{\supply_1,\demand_i\}$ for $i<j$ or $\{\supply_2,\demand_i\}$ for $i>j$.

In the first case, we have $s_{1i}-s_{2i} \geq  s_{1j}-s_{2j}$ and thus obtain
 reduced costs of $$(s_{1i}-s_{2i}) + (s_{2j}-s_{1j}) \geq 0.$$ In the second case, we have $s_{1i}-s_{2i} \leq  s_{1j}-s_{2j}$ and thus reduced costs $$(s_{2i}-s_{1i}) + (s_{1j}-s_{2j}) \geq 0.$$
\end{proof}


\begin{lemma}\label{lem:intnondecreasing}
Let $s\in \mathbb{R}^{2n}$, let $O\neq F$, let $y^F$ be maximal for $s^Ty$ and let finally $D_m^O\neq D_m^F$. Then there is some $\supply_i$ for which we may apply Lemma \ref{marklem} and the corresponding pivot will be non-decreasing. 
\end{lemma}

\begin{proof}
Let again, without loss of generality, $s\in \mathbb{R}^{2n}$ satisfy $s_{1i}-s_{2i}\geq s_{1(i+1)}-s_{2(i+1)}$ for all $i\leq n-1$. We have $D_m^O=\{\demand_q\}$, while $D_m^F=\{\demand_j\}$ for $j \neq q$; let us first consider the case $q<j$.

Note that $q<j$ implies that $\{\supply_2,\demand_q\} \notin F$, so that we may apply Lemma \ref{marklem} to $\supply_2$. Thus, when we apply a pivot, we insert an edge $\{\supply_2,\demand_p\}\in F\backslash O$ for which $p\geq q$. Such a pivot has reduced costs 
$$(s_{2p}-s_{1p})+(s_{1q}-s_{2q})\geq 0,$$
and thus a corresponding pivot, as determined using Lemma \ref{marklem}, will be non-decreasing. The case $q>j$ follows analogously with the roles of $\supply_1$ and $\supply_2$ switched around.
\end{proof}

Finally, we obtain the desired statement.

\begin{lemma}\label{thm:monotonehirsch}
The $2{\times}n$--transportation polytope has a monotone diameter of at most $n$.
\end{lemma}

\begin{proof} For a given $s\in \mathbb{R}^{2n}$, let $F$ be the corresponding maximal assignment and $O$ be the original assignment. By Lemma \ref{thm2n}, it is possible to arrive at $F$ after at most $n$ pivot steps -- using the approach outlined in its proof. It suffices to see that one can use non-decreasing pivots in this approach.

 If we have $D_m^O=D_m^F$, then this follows by Lemma \ref{lem:fracnondecreasing}, as then all pivots inserting edges in $F\backslash O$ are non-decreasing. Else, by Lemma \ref{lem:intnondecreasing} there is a non-decreasing next pivot that adheres to the process in the proof of Lemma \ref{thm2n}. This proves the claim.
\end{proof}

\subsection{$3{\times}n$--transportation polytope}

The Hirsch conjecture for the $3{\times}n$--transportation polytope claims a bound of $n+2-k$ on the diameter, where $k$ is the number of critical edges in the polytope. Before turning to the details, let us give a top-level view on the proof of this bound.
\\\\
\noindent \emph{Proof of Theorem \ref{thm:3n}.} Let $P$ be a non-degenerate $3{\times}n$ transportation polytope with $k$ critical edges and let $O$ be some initial assignment and $F$ be a final assignment. In subsection \ref{subsection3n} we present an algorithm such that at each step we choose an edge in $F$ to insert (if needed) and then mark it, such that it satisfies the conditions of our marking system, and no marked edge is ever deleted. Since a marked edge is never deleted, once we do this for all $|F|$ edges we reach our final assignment. This requires at most $|F|-k=3+n-k-1=n+2-k$ insertions, and thus pivots, since there are $|F|$ edges to mark and each one requires at most $1$ insertion, and we also know that when the critical edges are marked they need not be inserted since they exist in every assignment. Therefore since the edge distance between any two assignments is at most $n+2-k$, we have the diameter is at most the Hirsch bound. \qed
\\\\

The cusp of this argument is the lengthy algorithm for choosing which edge to mark (after possibly inserting it) in the following subsection. Like the proof of Lemma \ref{thm2n}, we consider a case-by-case analysis of the marked and unmarked edge combinations in the possible configurations for the mixed edges in $O$. Also, like Lemma \ref{thm2n}, we utilize the power of Lemma \ref{marklem} whenever possible. However, there are several cases where this cannot be used, which complicates the algorithm. Further we must be wary of the following two marked mixed edge configurations.
      
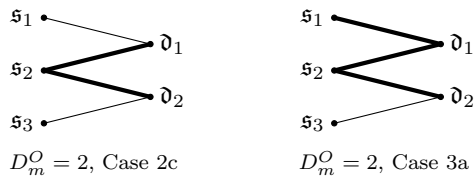
\begin{figure}[H]
        \centering
\subfloat[$D_m^O=2$, Case 2c]{
            \begin{tikzpicture}[scale=0.35]
                    \coordinate (c1) at (0,4);
                    \coordinate (c2) at (0,2);
                    \coordinate (c3) at (0,0);
                    \coordinate (x1) at (4,3);
                    \coordinate (x2) at (4,1);
                    \draw [fill, black] (c1) circle (0.1cm);
                    \draw [fill, black] (c2) circle (0.1cm);
                    \draw [fill, black] (c3) circle (0.1cm);
                    \draw [fill, black] (x1) circle (0.1cm);  
                    \draw [fill, black] (x2) circle (0.1cm);
                    \node[left] at (c1) {$\supply_1$};
                    \node[left] at (c2) {$\supply_2$};
                    \node[left] at (c3) {$\supply_3$};
                    \node[right] at (x1) {$\demand_1$};
                    \node[right] at (x2) {$\demand_2$};
                    \draw(c1)--(x1);
                    \draw[ultra thick] (c2)--(x1);
                    \draw[ultra thick] (c2)--(x2);
                    \draw (c3)--(x2);
            \end{tikzpicture}}
\hspace{3em}
\subfloat[$D_m^O=2$, Case 3a]{
            \begin{tikzpicture}[scale=0.35]
                    \coordinate (c1) at (0,4);
                    \coordinate (c2) at (0,2);
                    \coordinate (c3) at (0,0);
                    \coordinate (x1) at (4,3);
                    \coordinate (x2) at (4,1);
                    \draw [fill, black] (c1) circle (0.1cm);
                    \draw [fill, black] (c2) circle (0.1cm);
                    \draw [fill, black] (c3) circle (0.1cm);
                    \draw [fill, black] (x1) circle (0.1cm);
                    \draw [fill, black] (x2) circle (0.1cm);
                    \node[left] at (c1) {$\supply_1$};
                    \node[left] at (c2) {$\supply_2$};
                    \node[left] at (c3) {$\supply_3$};
                    \node[right] at (x1) {$\demand_1$};
                    \node[right] at (x2) {$\demand_2$};
                    \draw[ultra thick](c1)--(x1);
                    \draw[ultra thick] (c2)--(x1);
                    \draw[ultra thick] (c2)--(x2);
                    \draw (c3)--(x2);
            \end{tikzpicture}}
        \end{figure}
We avoid the first of these cases ($D_m^O=2$, Case 2c) altogether, and we will require special restrictions on $\demand_2$ when we enter the second of these cases ($D_m^O=2$, Case 3a).


\subsection{Algorithm for the proof of Theorem \ref{thm:3n}}\label{subsection3n}

For the sake of a simple wording, we always refer to the current assignment as $O$ in our algorithm, and we assume that initially we start out with a copy of our initial assignment with no edges marked. We structure our investigation of the configurations by the number of mixed demands $|D_m^O|$, which is equal to $1$ or $2$, and the number of marked edges incident to $D_m^O$.

\subsubsection*{Case $|D_m^O|=1$}

    
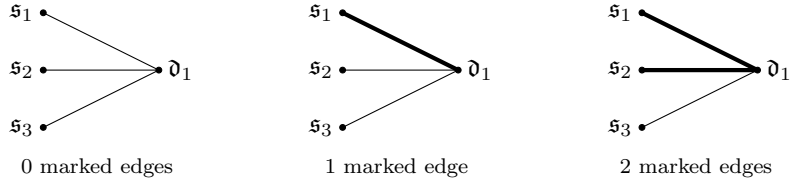
\begin{figure}[H]
        \centering
\subfloat[$0$ marked edges]{
            \begin{tikzpicture}[scale=0.38]
                    \coordinate (c1) at (0,4);
                    \coordinate (c2) at (0,2);
                    \coordinate (c3) at (0,0);
                    \coordinate (x1) at (4,2);
                    \draw [fill, black] (c1) circle (0.1cm);
                    \draw [fill, black] (c2) circle (0.1cm);
                    \draw [fill, black] (c3) circle (0.1cm);
                    \draw [fill, black] (x1) circle (0.1cm);
                    \node[left] at (c1) {$\supply_1$};
                    \node[left] at (c2) {$\supply_2$};
                    \node[left] at (c3) {$\supply_3$};
                    \node[right] at (x1) {$\demand_1$};
                    \draw(c1)--(x1);
                    \draw(c2) --(x1);
                    \draw(c3) --(x1);
            \end{tikzpicture}}
\hspace{3em}
 \subfloat[$1$ marked edge]{
            \begin{tikzpicture}[scale=0.38]
                    \coordinate (c1) at (0,4);
                    \coordinate (c2) at (0,2);
                    \coordinate (c3) at (0,0);
                    \coordinate (x1) at (4,2);
                    \draw [fill, black] (c1) circle (0.1cm);
                    \draw [fill, black] (c2) circle (0.1cm);
                    \draw [fill, black] (c3) circle (0.1cm);
                    \draw [fill, black] (x1) circle (0.1cm);
                    \node[left] at (c1) {$\supply_1$};
                    \node[left] at (c2) {$\supply_2$};
                    \node[left] at (c3) {$\supply_3$};
                    \node[right] at (x1) {$\demand_1$};
                    \draw[ultra thick] (c1)--(x1);
                    \draw(c2) --(x1);
                    \draw(c3) --(x1);
            \end{tikzpicture}}
\hspace{3em}
\subfloat[$2$ marked edges]{
            \begin{tikzpicture}[scale=0.38]
                    \coordinate (c1) at (0,4);
                    \coordinate (c2) at (0,2);
                    \coordinate (c3) at (0,0);
                    \coordinate (x1) at (4,2);
                    \draw [fill, black] (c1) circle (0.1cm);
                    \draw [fill, black] (c2) circle (0.1cm);
                    \draw [fill, black] (c3) circle (0.1cm);
                    \draw [fill, black] (x1) circle (0.1cm);
                    \node[left] at (c1) {$\supply_1$};
                    \node[left] at (c2) {$\supply_2$};
                    \node[left] at (c3) {$\supply_3$};
                    \node[right] at (x1) {$\demand_1$};
                    \draw[ultra thick] (c1)--(x1);
                    \draw[ultra thick] (c2) --(x1);
                    \draw (c3) --(x1);
            \end{tikzpicture}}

        \end{figure}
\noindent If our mixed edges look like any of these cases, there is some $\supply_i$ such that we can apply Lemma \ref{marklem} and mark an edge after (at most) one pivot. Otherwise our mixed edges look like:
    
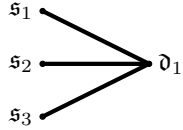
\begin{figure}[H]
        \centering
            \begin{tikzpicture}[scale=0.35]
                    \coordinate (c1) at (0,4);
                    \coordinate (c2) at (0,2);
                    \coordinate (c3) at (0,0);
                    \coordinate (x1) at (4,2);
                    \draw [fill, black] (c1) circle (0.1cm);
                    \draw [fill, black] (c2) circle (0.1cm);
                    \draw [fill, black] (c3) circle (0.1cm);
                    \draw [fill, black] (x1) circle (0.1cm);
                    \node[left] at (c1) {$\supply_1$};
                    \node[left] at (c2) {$\supply_2$};
                    \node[left] at (c3) {$\supply_3$};
                    \node[right] at (x1) {$\demand_1$};
                    \draw[ultra thick] (c1)--(x1);
                    \draw[ultra thick] (c2) --(x1);
                    \draw[ultra thick] (c3) --(x1);
            \end{tikzpicture}
        \end{figure}
\noindent Then, by assumption, all $\supply_i$ have all their leaf edges and $E_m^O=E_m^F$. Hence $O=F$.
\bigskip

\bigskip

 \subsubsection*{Case $|D_m^O|=2$}

\noindent \textbf{0 or 1 marked edges}
        
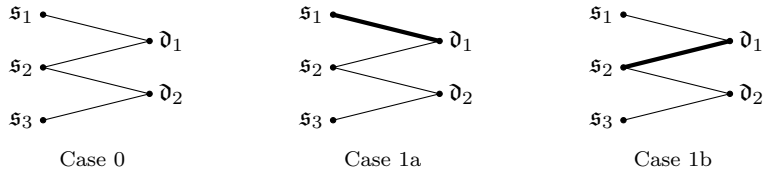
\begin{figure}[H]
        \centering
\subfloat[Case 0]{
            \begin{tikzpicture}[scale=0.35]
                    \coordinate (c1) at (0,4);
                    \coordinate (c2) at (0,2);
                    \coordinate (c3) at (0,0);
                    \coordinate (x1) at (4,3);
                    \coordinate (x2) at (4,1);
                    \draw [fill, black] (c1) circle (0.1cm);
                    \draw [fill, black] (c2) circle (0.1cm);
                    \draw [fill, black] (c3) circle (0.1cm);
                    \draw [fill, black] (x1) circle (0.1cm);
                    \draw [fill, black] (x2) circle (0.1cm);
                    \node[left] at (c1) {$\supply_1$};
                    \node[left] at (c2) {$\supply_2$};
                    \node[left] at (c3) {$\supply_3$};
                    \node[right] at (x1) {$\demand_1$};
                    \node[right] at (x2) {$\demand_2$};
                    \draw(c1)--(x1);
                    \draw (c2)--(x1);
                    \draw (c2)--(x2);
                    \draw (c3)--(x2);
            \end{tikzpicture}}
\hspace{3em}
\subfloat[Case 1a]{
            \begin{tikzpicture}[scale=0.35]
                    \coordinate (c1) at (0,4);
                    \coordinate (c2) at (0,2);
                    \coordinate (c3) at (0,0);
                    \coordinate (x1) at (4,3);
                    \coordinate (x2) at (4,1);
                    \draw [fill, black] (c1) circle (0.1cm);
                    \draw [fill, black] (c2) circle (0.1cm);
                    \draw [fill, black] (c3) circle (0.1cm);
                    \draw [fill, black] (x1) circle (0.1cm);
                    \draw [fill, black] (x2) circle (0.1cm);
                    \node[left] at (c1) {$\supply_1$};
                    \node[left] at (c2) {$\supply_2$};
                    \node[left] at (c3) {$\supply_3$};
                    \node[right] at (x1) {$\demand_1$};
                    \node[right] at (x2) {$\demand_2$};
                    \draw[ultra thick] (c1)--(x1);
                    \draw (c2)--(x1);
                    \draw (c2)--(x2);
                    \draw (c3)--(x2);
            \end{tikzpicture}}
\hspace{3em}
\subfloat[Case 1b]{
            \begin{tikzpicture}[scale=0.35]
                    \coordinate (c1) at (0,4);
                    \coordinate (c2) at (0,2);
                    \coordinate (c3) at (0,0);
                    \coordinate (x1) at (4,3);
                    \coordinate (x2) at (4,1);
                    \draw [fill, black] (c1) circle (0.1cm);
                    \draw [fill, black] (c2) circle (0.1cm);
                    \draw [fill, black] (c3) circle (0.1cm);
                    \draw [fill, black] (x1) circle (0.1cm);
                    \draw [fill, black] (x2) circle (0.1cm);
                    \node[left] at (c1) {$\supply_1$};
                    \node[left] at (c2) {$\supply_2$};
                    \node[left] at (c3) {$\supply_3$};
                    \node[right] at (x1) {$\demand_1$};
                    \node[right] at (x2) {$\demand_2$};
                    \draw (c1)--(x1);
                    \draw[ultra thick] (c2)--(x1);
                    \draw (c2)--(x2);
                    \draw (c3)--(x2);
            \end{tikzpicture}}
       \end{figure}
\noindent Clearly we may apply Lemma \ref{marklem} to $\supply_3$, and mark an edge after (at most) one pivot.
\bigskip

\noindent \textbf{2 marked edges}
        \bigskip

\noindent \textbf{Case 2a.\qquad }
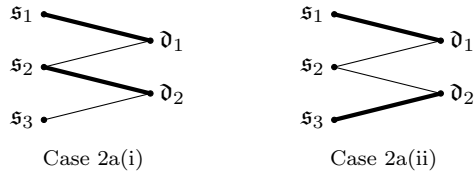
     \begin{figure}[H]
        \centering
\subfloat[Case 2a(i)]{
            \begin{tikzpicture}[scale=0.35]
                    \coordinate (c1) at (0,4);
                    \coordinate (c2) at (0,2);
                    \coordinate (c3) at (0,0);
                    \coordinate (x1) at (4,3);
                    \coordinate (x2) at (4,1);
                    \draw [fill, black] (c1) circle (0.1cm);
                    \draw [fill, black] (c2) circle (0.1cm);
                    \draw [fill, black] (c3) circle (0.1cm);
                    \draw [fill, black] (x1) circle (0.1cm);
                    \draw [fill, black] (x2) circle (0.1cm);
                    \node[left] at (c1) {$\supply_1$};
                    \node[left] at (c2) {$\supply_2$};
                    \node[left] at (c3) {$\supply_3$};
                    \node[right] at (x1) {$\demand_1$};
                    \node[right] at (x2) {$\demand_2$};
                    \draw[ultra thick] (c1)--(x1);
                    \draw (c2)--(x1);
                    \draw[ultra thick] (c2)--(x2);
                    \draw (c3)--(x2);
            \end{tikzpicture}}
\hspace{3em}
\subfloat[Case 2a(ii)]{
            \begin{tikzpicture}[scale=0.35]
                    \coordinate (c1) at (0,4);
                    \coordinate (c2) at (0,2);
                    \coordinate (c3) at (0,0);
                    \coordinate (x1) at (4,3);
                    \coordinate (x2) at (4,1);
                    \draw [fill, black] (c1) circle (0.1cm);
                    \draw [fill, black] (c2) circle (0.1cm);
                    \draw [fill, black] (c3) circle (0.1cm);
                    \draw [fill, black] (x1) circle (0.1cm);
                    \draw [fill, black] (x2) circle (0.1cm);
                    \node[left] at (c1) {$\supply_1$};
                    \node[left] at (c2) {$\supply_2$};
                    \node[left] at (c3) {$\supply_3$};
                    \node[right] at (x1) {$\demand_1$};
                    \node[right] at (x2) {$\demand_2$};
                    \draw[ultra thick] (c1)--(x1);
                    \draw (c2)--(x1);
                    \draw (c2)--(x2);
                    \draw[ultra thick] (c3)--(x2);
            \end{tikzpicture}}
       \end{figure}
We can apply Lemma \ref{marklem} to $\supply_3$ in case 2a(i) and $\supply_2$ in case 2b(ii), and mark one edge after (at most) one pivot.
\bigskip

\noindent \textbf{Case 2b.\qquad }
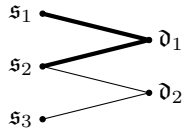
        \begin{figure}[H]
        \centering
            \begin{tikzpicture}[scale=0.35]
                    \coordinate (c1) at (0,4);
                    \coordinate (c2) at (0,2);
                    \coordinate (c3) at (0,0);
                    \coordinate (x1) at (4,3);
                    \coordinate (x2) at (4,1);
                    \draw [fill, black] (c1) circle (0.1cm);
                    \draw [fill, black] (c2) circle (0.1cm);
                    \draw [fill, black] (c3) circle (0.1cm);
                    \draw [fill, black] (x1) circle (0.1cm);
                    \draw [fill, black] (x2) circle (0.1cm);
                    \node[left] at (c1) {$\supply_1$};
                    \node[left] at (c2) {$\supply_2$};
                    \node[left] at (c3) {$\supply_3$};
                    \node[right] at (x1) {$\demand_1$};
                    \node[right] at (x2) {$\demand_2$};
                    \draw[ultra thick] (c1)--(x1);
                    \draw[ultra thick] (c2)--(x1);
                    \draw (c2)--(x2);
                    \draw (c3)--(x2);
            \end{tikzpicture}
       \end{figure}

\noindent We do the following.
\begin{enumerate}
        \item If $\demand_2$ is a leaf demand in $F$, we know that $\{\supply_3,\demand_2\}\in F$ ($\supply_1$ and  $\supply_2$ already have all their leaf edges since they are incident to marked mixed edges), we mark it.
    \item Else $\demand_2$ is a mixed demand in $F$. If $\{\supply_2,\demand_2\}\in F$, we mark it.

          Otherwise $\{\supply_1,\demand_2\},\{\supply_3,\demand_2\}\in F$. We insert $\{\supply_1,\demand_2\}$:
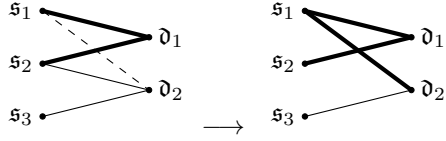
   \begin{figure}[H]
        \centering
        \subfloat{
            \begin{tikzpicture}[scale=0.35]
                    \coordinate (c1) at (0,4);
                    \coordinate (c2) at (0,2);
                    \coordinate (c3) at (0,0);
                    \coordinate (x1) at (4,3);
                    \coordinate (x2) at (4,1);
                    \draw [fill, black] (c1) circle (0.1cm);
                    \draw [fill, black] (c2) circle (0.1cm);
                    \draw [fill, black] (c3) circle (0.1cm);
                    \draw [fill, black] (x1) circle (0.1cm);
                    \draw [fill, black] (x2) circle (0.1cm);
                    \node[left] at (c1) {$\supply_1$};
                    \node[left] at (c2) {$\supply_2$};
                    \node[left] at (c3) {$\supply_3$};
                    \node[right] at (x1) {$\demand_1$};
                    \node[right] at (x2) {$\demand_2$};
                    \draw[ultra thick] (c1)--(x1);
                    \draw[ultra thick] (c2)--(x1);
                    \draw (c2)--(x2);
                    \draw  (c3)--(x2);
                     \draw[dashed] (c1)--(x2);
                            \end{tikzpicture} }
            $\longrightarrow$
            \subfloat{
                            \begin{tikzpicture}[scale=0.35]
                    \coordinate (c1) at (0,4);
                    \coordinate (c2) at (0,2);
                    \coordinate (c3) at (0,0);
                    \coordinate (x1) at (4,3);
                    \coordinate (x2) at (4,1);
                    \draw [fill, black] (c1) circle (0.1cm);
                    \draw [fill, black] (c2) circle (0.1cm);
                    \draw [fill, black] (c3) circle (0.1cm);
                    \draw [fill, black] (x1) circle (0.1cm);
                    \draw [fill, black] (x2) circle (0.1cm);
                    \node[left] at (c1) {$\supply_1$};
                    \node[left] at (c2) {$\supply_2$};
                    \node[left] at (c3) {$\supply_3$};
                    \node[right] at (x1) {$\demand_1$};
                    \node[right] at (x2) {$\demand_2$};
                     \draw[ultra thick] (c1)--(x1);
                    \draw[ultra thick] (c2)--(x1);
                    \draw (c3)--(x2);
                    \draw[ultra thick] (c1)--(x2);
            \end{tikzpicture} }
       \end{figure}
       This pivot deletes $\{\supply_2,\demand_2\}$ (otherwise there is an edge incident to $\supply_2$ to delete, but no edge to increase).
\end{enumerate}
 \bigskip

\noindent \textbf{Case 2c.\qquad }
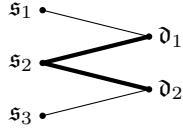
        \begin{figure}[H]
        \centering
            \begin{tikzpicture}[scale=0.35]
                    \coordinate (c1) at (0,4);
                    \coordinate (c2) at (0,2);
                    \coordinate (c3) at (0,0);
                    \coordinate (x1) at (4,3);
                    \coordinate (x2) at (4,1);
                    \draw [fill, black] (c1) circle (0.1cm);
                    \draw [fill, black] (c2) circle (0.1cm);
                    \draw [fill, black] (c3) circle (0.1cm);
                    \draw [fill, black] (x1) circle (0.1cm);
                    \draw [fill, black] (x2) circle (0.1cm);
                    \node[left] at (c1) {$\supply_1$};
                    \node[left] at (c2) {$\supply_2$};
                    \node[left] at (c3) {$\supply_3$};
                    \node[right] at (x1) {$\demand_1$};
                    \node[right] at (x2) {$\demand_2$};
                    \draw (c1)--(x1);
                    \draw[ultra thick] (c2)--(x1);
                    \draw[ultra thick] (c2)--(x2);
                    \draw (c3)--(x2);
            \end{tikzpicture}
       \end{figure}
\noindent This case will never occur when we follow the algorithm given in this proof. Note that in each case the edge marked is either the one inserted or it already exists. Therefore, if we were to enter case 2c, one of these marked edges would have already been marked in the previous assignment and would have been in the mixed part of this assignment.

	Thus, since applying Lemma \ref{marklem} always marks an edge incident to some $\supply_i$ which is not already incident to some marked edge in the mixed part of the assignment, we will never end up in case 2c. Hence we do not enter this case from cases $|D_m^O|=1$ or $|D_M^O|=2$ for 0 marked edges, 1 marked edge, or case 2a.

Note also, once there are two marked edges incident to the same $\demand_j$, they will forever be marked and in the mixed part of the assignment, hence we could not arrive in case 2c from case 2b or any of the 3 marked edge cases.
\bigskip

\noindent \textbf{3 marked edges}
\bigskip

\noindent \textbf{Case 3a.\qquad }
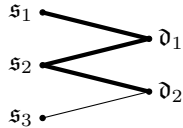
        \begin{figure}[H]
        \centering
            \begin{tikzpicture}[scale=0.35]
                    \coordinate (c1) at (0,4);
                    \coordinate (c2) at (0,2);
                    \coordinate (c3) at (0,0);
                    \coordinate (x1) at (4,3);
                    \coordinate (x2) at (4,1);
                    \draw [fill, black] (c1) circle (0.1cm);
                    \draw [fill, black] (c2) circle (0.1cm);
                    \draw [fill, black] (c3) circle (0.1cm);
                    \draw [fill, black] (x1) circle (0.1cm);
                    \draw [fill, black] (x2) circle (0.1cm);
                    \node[left] at (c1) {$\supply_1$};
                    \node[left] at (c2) {$\supply_2$};
                    \node[left] at (c3) {$\supply_3$};
                    \node[right] at (x1) {$\demand_1$};
                    \node[right] at (x2) {$\demand_2$};
                    \draw[ultra thick] (c1)--(x1);
                    \draw[ultra thick] (c2)--(x1);
                    \draw[ultra thick] (c2)--(x2);
                    \draw (c3)--(x2);
            \end{tikzpicture}
       \end{figure}

First note when we enter this configuration that $\demand_2$ must be a mixed demand in $F$. For this assume we have entered case 3a from a different case. At the end of this case we will show that once we are in case 3a we either remain in the case (with the same mixed edges and thus $\demand_2$ remains a mixed demand), or we enter either case 3b or the case with 4 marked edges.

As argued in case 2c, if we enter case 3a, at least one of the edges $\{\supply_2,\demand_1\}$ or $\{\supply_2,\demand_2\}$ must have been a mixed edge in the previous assignment and marked. Hence, since Lemma \ref{marklem} always marks an edge incident to some $\supply_i$ with no edges marked in the mixed part of the assignment, we will never enter case 3a when applying Lemma \ref{marklem}. Therefore we will never enter this case after applying cases $|D_m^O|=1$ or $|D_M^O|=2$ for 0 marked edges, 1 marked edge, or case 2a.

Then for case 2b(1), we will enter case 3b. For case 2b(2) we will enter case 3a, however note that  $\demand_2$ is the same node in both cases and is mixed (note that most node labelings in mixed part remain the same, although $\supply_1$ and $\supply_2$ may switch their labels). Finally, for case 3b, we have that either we will enter the case with 4 marked edges, or we will enter case 3a. However when we enter case 3a (this happens in case 3b(2) if $\{\supply_3,\demand_3\}$ is unmarked), we have that $\demand_3$ is a mixed demand in $F$.\\

\noindent Using this knowledge, we do the following:
    \begin{enumerate}
	\item If there is an unmarked leaf edge $\{\supply_3,\demand_3\}$, with $\{\supply_2,\demand_3\}\in O$, we insert  $\{\supply_3,\demand_3\}$ (if necessary) and mark it. Clearly no marked edges are deleted.

\item  Else if there is a leaf edge of the form $\{\supply_3,\demand_3\}$ to insert, we insert and mark this edge. 
            
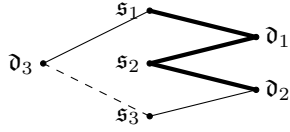
\begin{figure}[H]
         \centering
            \begin{tikzpicture}[scale=0.35]
                    \coordinate (c1) at (0,4);
                    \coordinate (c2) at (0,2);
                    \coordinate (c3) at (0,0);
                    \coordinate (x1) at (4,3);
                    \coordinate (x2) at (4,1);
                    \coordinate (x3) at (-4,2);
                    \draw [fill, black] (c1) circle (0.1cm);
                    \draw [fill, black] (c2) circle (0.1cm);
                    \draw [fill, black] (c3) circle (0.1cm);
                    \draw [fill, black] (x1) circle (0.1cm);
                    \draw [fill, black] (x2) circle (0.1cm);
                    \draw [fill, black] (x3) circle (0.1cm);
                    \node[left] at (c1) {$\supply_1$};
                    \node[left] at (c2) {$\supply_2$};
                    \node[left] at (c3) {$\supply_3$};
                    \node[right] at (x1) {$\demand_1$};
                    \node[right] at (x2) {$\demand_2$};
                    \node[left] at (x3) {$\demand_3$};
                    \draw[ultra thick] (c1)--(x1);
                    \draw[ultra thick] (c2)--(x1);
                    \draw[ultra thick] (c2)--(x2);
                    \draw (c1)--(x3);
                    \draw[dashed] (c3)--(x3);
                    \draw (c3)--(x2);
            \end{tikzpicture}
       \end{figure}
          This cannot delete $\{\supply_2,\demand_1\}$. Suppose it does, and let $C$ be the assignment obtained by deleting $\{\supply_2,\demand_1\}$. In $C$ we would need to decrease $\{\supply_1,\demand_1\}$, and thus there is some edge $\{\supply_1,\demand_j\}$ incident to $\supply_1$ to increase. However since $\supply_1$ already has all its leaf edges and they are marked and leaves in $O$, this would imply that $\{\supply_1,\demand_j\}$ is mixed in $F$. However, since $\demand_1$ and $\demand_2$ are the only mixed demands in $F$and clearly $\demand_2\neq \demand_j$, this is a contradiction. 

\item Else, since $\demand_2$ is a mixed demand in $F$, we have that $\{\supply_3,\demand_2\}\in F$ so we mark it (since $\supply_3$ has all its leaf edges)
\end{enumerate}

Finally, as stated at the beginning of the case, we want to note what case we enter after these steps. If step (3) is applied we  enter the case with 4 marked edges. Consider the three configurations below. If step (1) is applied, we arrive at the left or middle one, if step (2) is applied, we arrive at the middle or right one.

      
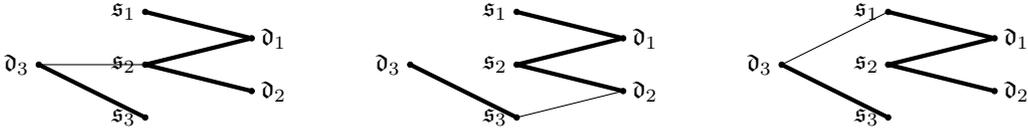
\begin{figure}[H]
         \centering
            \begin{tikzpicture}[scale=0.35]
                    \coordinate (c1) at (0,4);
                    \coordinate (c2) at (0,2);
                    \coordinate (c3) at (0,0);
                    \coordinate (x1) at (4,3);
                    \coordinate (x2) at (4,1);
                    \coordinate (x3) at (-4,2);
                    \draw [fill, black] (c1) circle (0.1cm);
                    \draw [fill, black] (c2) circle (0.1cm);
                    \draw [fill, black] (c3) circle (0.1cm);
                    \draw [fill, black] (x1) circle (0.1cm);
                    \draw [fill, black] (x2) circle (0.1cm);
                     \draw [fill, black] (x3) circle (0.1cm);
                    \node[left] at (c1) {$\supply_1$};
                    \node[left] at (c2) {$\supply_2$};
                    \node[left] at (c3) {$\supply_3$};
                    \node[right] at (x1) {$\demand_1$};
                    \node[right] at (x2) {$\demand_2$};
                     \node[left] at (x3) {$\demand_3$};
                    \draw[ultra thick] (c1)--(x1);
                    \draw[ultra thick] (c2)--(x1);
                    \draw[ultra thick] (c2)--(x2);
                     \draw (c2)--(x3);
		\draw[ultra thick] (c3)--(x3);
            \end{tikzpicture}
\hspace{2em}
\begin{tikzpicture}[scale=0.35]
                    \coordinate (c1) at (0,4);
                    \coordinate (c2) at (0,2);
                    \coordinate (c3) at (0,0);
                    \coordinate (x1) at (4,3);
                    \coordinate (x2) at (4,1);
                    \coordinate (x3) at (-4,2);
                    \draw [fill, black] (c1) circle (0.1cm);
                    \draw [fill, black] (c2) circle (0.1cm);
                    \draw [fill, black] (c3) circle (0.1cm);
                    \draw [fill, black] (x1) circle (0.1cm);
                    \draw [fill, black] (x2) circle (0.1cm);
                     \draw [fill, black] (x3) circle (0.1cm);
                    \node[left] at (c1) {$\supply_1$};
                    \node[left] at (c2) {$\supply_2$};
                    \node[left] at (c3) {$\supply_3$};
                    \node[right] at (x1) {$\demand_1$};
                    \node[right] at (x2) {$\demand_2$};
                     \node[left] at (x3) {$\demand_3$};
                    \draw[ultra thick] (c1)--(x1);
                    \draw[ultra thick] (c2)--(x1);
                    \draw[ultra thick] (c2)--(x2);

		\draw (c3)--(x2);
                     \draw[ultra thick] (c3)--(x3);
            \end{tikzpicture} 
\hspace{2em} 
            \begin{tikzpicture}[scale=0.35]
                    \coordinate (c1) at (0,4);
                    \coordinate (c2) at (0,2);
                    \coordinate (c3) at (0,0);
                    \coordinate (x1) at (4,3);
                    \coordinate (x2) at (4,1);
                    \coordinate (x3) at (-4,2);
                    \draw [fill, black] (c1) circle (0.1cm);
                    \draw [fill, black] (c2) circle (0.1cm);
                    \draw [fill, black] (c3) circle (0.1cm);
                    \draw [fill, black] (x1) circle (0.1cm);
                    \draw [fill, black] (x2) circle (0.1cm);
                     \draw [fill, black] (x3) circle (0.1cm);
                    \node[left] at (c1) {$\supply_1$};
                    \node[left] at (c2) {$\supply_2$};
                    \node[left] at (c3) {$\supply_3$};
                    \node[right] at (x1) {$\demand_1$};
                    \node[right] at (x2) {$\demand_2$};
                     \node[left] at (x3) {$\demand_3$};
                    \draw[ultra thick] (c1)--(x1);
                    \draw[ultra thick] (c2)--(x1);
                    \draw[ultra thick] (c2)--(x2);
		\draw (c1)--(x3);
                     \draw[ultra thick] (c3)--(x3);
            \end{tikzpicture}    
       \end{figure}
In the first two figures we have a configuration of type case 3b. In the last figure we have a configuration that is of type case 3a, but our $\demand_2$ stayed the same and thus is still a mixed demand.
\bigskip

\noindent \textbf{Case 3b.\qquad }
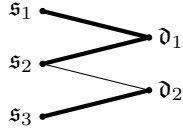
        \begin{figure}[H]
        \centering
            \begin{tikzpicture}[scale=0.35]
                    \coordinate (c1) at (0,4);
                    \coordinate (c2) at (0,2);
                    \coordinate (c3) at (0,0);
                    \coordinate (x1) at (4,3);
                    \coordinate (x2) at (4,1);
                    \draw [fill, black] (c1) circle (0.1cm);
                    \draw [fill, black] (c2) circle (0.1cm);
                    \draw [fill, black] (c3) circle (0.1cm);
                    \draw [fill, black] (x1) circle (0.1cm);
                    \draw [fill, black] (x2) circle (0.1cm);
                    \node[left] at (c1) {$\supply_1$};
                    \node[left] at (c2) {$\supply_2$};
                    \node[left] at (c3) {$\supply_3$};
                    \node[right] at (x1) {$\demand_1$};
                    \node[right] at (x2) {$\demand_2$};
                    \draw[ultra thick] (c1)--(x1);
                    \draw[ultra thick] (c2)--(x1);
                    \draw (c2)--(x2);
                    \draw[ultra thick] (c3)--(x2);
            \end{tikzpicture}
       \end{figure}

\noindent We do the following.
\begin{enumerate}
    \item If $\demand_2$ is a mixed demand in $F$, we mark $\{\supply_2,\demand_2\}$ or insert and mark $\{\supply_1,\demand_2\}$ (this pivot deletes $\{\supply_2,\demand_2\}$, as otherwise there would be an edge incident to $\supply_2$ to delete, but no edge to increase). Note we can do these markings since $\supply_1$ and $\supply_2$ necessarily have all their leaf edges. 

    \item Else $\demand_2$ is a leaf demand in $F$. Note there are no edges $\{\supply_3,\demand_3\}$ incident to $\supply_3$ to delete, as $\demand_3$ would be a leaf node in $F$, but $\supply_1$ and $\supply_2$ already have all their leaf edges. As there are no more edges incident to $\supply_3$ to delete, we must have some leaf edge $\{\supply_3,\demand_3\}\in O$ to decrease which is mixed in $F$. 

We insert the other mixed edge $\{\supply_1,\demand_3\}$ or $\{\supply_2,\demand_3\}$ and mark it.
        
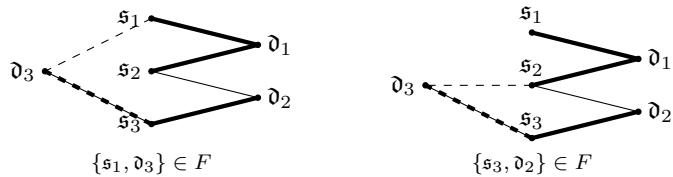
\begin{figure}[H]
        \centering
                \subfloat[$\{\supply_1,\demand_3\}\in F$]{
            \begin{tikzpicture}[scale=0.35]
                    \coordinate (c1) at (0,4);
                    \coordinate (c2) at (0,2);
                    \coordinate (c3) at (0,0);
                    \coordinate (x1) at (4,3);
                    \coordinate (x2) at (4,1);
                    \coordinate (x3) at (-4,2);
                    \draw [fill, black] (c1) circle (0.1cm);
                    \draw [fill, black] (c2) circle (0.1cm);
                    \draw [fill, black] (c3) circle (0.1cm);
                    \draw [fill, black] (x1) circle (0.1cm);
                    \draw [fill, black] (x2) circle (0.1cm);
                    \draw [fill, black] (x3) circle (0.1cm);
                    \node[left] at (c1) {$\supply_1$};
                    \node[left] at (c2) {$\supply_2$};
                    \node[left] at (c3) {$\supply_3$};
                    \node[right] at (x1) {$\demand_1$};
                    \node[right] at (x2) {$\demand_2$};
                     \node[left] at (x3) {$\demand_3$};
                    \draw[ultra thick] (c1)--(x1);
                    \draw[ultra thick] (c2)--(x1);
                    \draw (c2)--(x2);
                    \draw[ultra thick] (c3)--(x2);
                    \draw[dashed] (c1)--(x3);
                    \draw (c3)--(x3);
		\draw[ultra thick, dashed] (c3)--(x3);
            \end{tikzpicture}
               }
               \hspace{2em}
               \subfloat[$\{\supply_3,\demand_2\}\in F$]{
            \begin{tikzpicture}[scale=0.35]
                    \coordinate (c1) at (0,4);
                    \coordinate (c2) at (0,2);
                    \coordinate (c3) at (0,0);
                    \coordinate (x1) at (4,3);
                    \coordinate (x2) at (4,1);
                    \coordinate (x3) at (-4,2);
                    \draw [fill, black] (c1) circle (0.1cm);
                    \draw [fill, black] (c2) circle (0.1cm);
                    \draw [fill, black] (c3) circle (0.1cm);
                    \draw [fill, black] (x1) circle (0.1cm);
                    \draw [fill, black] (x2) circle (0.1cm);
                    \draw [fill, black] (x3) circle (0.1cm);
                    \node[above] at (c1) {$\supply_1$};
                    \node[above] at (c2) {$\supply_2$};
                    \node[above] at (c3) {$\supply_3$};
                    \node[right] at (x1) {$\demand_1$};
                    \node[right] at (x2) {$\demand_2$};
                    \node[left] at (x3) {$\demand_3$};
                    \draw[ultra thick] (c1)--(x1);
                    \draw[ultra thick] (c2)--(x1);
                    \draw (c2)--(x2);
                    \draw[ultra thick] (c3)--(x2);
                    \draw (c3)--(x3);
		\draw[dashed, ultra thick] (c3)--(x3);
                    \draw[dashed] (c2)--(x3);
            \end{tikzpicture}
                    }
       \end{figure}
\noindent If $\{\supply_3,\demand_3\}$ is marked in $O$, then this is the last edge to insert ($|O\setminus F|=1$) so that no marked edges are deleted in the pivot and we end up in $F$. Otherwise  $\{\supply_3,\demand_3\}$ is unmarked. Then if we are inserting $\{\supply_2,\demand_3\}$ clearly no marked edges are deleted. If instead we are inserting $\{\supply_1,\demand_3\}$, we know that $\{\supply_1,\demand_1\}$ is not deleted as otherwise $\supply_2$ would have two edges to decrease, but none to increase.
\end{enumerate}
\bigskip

\noindent \textbf{4 marked edges}
       
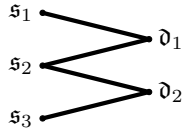
\begin{figure}[H]
        \centering
            \begin{tikzpicture}[scale=0.35]
                    \coordinate (c1) at (0,4);
                    \coordinate (c2) at (0,2);
                    \coordinate (c3) at (0,0);
                    \coordinate (x1) at (4,3);
                    \coordinate (x2) at (4,1);
                    \draw [fill, black] (c1) circle (0.1cm);
                    \draw [fill, black] (c2) circle (0.1cm);
                    \draw [fill, black] (c3) circle (0.1cm);
                    \draw [fill, black] (x1) circle (0.1cm);
                    \draw [fill, black] (x2) circle (0.1cm);
                    \node[left] at (c1) {$\supply_1$};
                    \node[left] at (c2) {$\supply_2$};
                    \node[left] at (c3) {$\supply_3$};
                    \node[right] at (x1) {$\demand_1$};
                    \node[right] at (x2) {$\demand_2$};
                    \draw[ultra thick] (c1)--(x1);
                    \draw[ultra thick] (c2)--(x1);
                    \draw[ultra thick] (c2)--(x2);
                    \draw[ultra thick] (c3)--(x2);
            \end{tikzpicture}
       \end{figure}
\noindent By assumption, all $\supply_i$ have all their leaf edges and $E_m^O=E_m^F$. Hence $O=F$.\\\\
\bigskip
This concludes the proof of Theorem \ref{thm:3n}.

\section*{Acknowledgements} 

Borgwardt gratefully acknowledges support from the Alexander von Humboldt Foundation, De Loera and Miller gratefully acknowledge support from UC MEXUS grant, and Finhold gratefully acknowledges support from the graduate program TopMath of the Elite Network of Bavaria and the TopMath Graduate Center of TUM Graduate School at Technische Universit\"at M\"unchen.


\begin{thebibliography}{99}
\bibitem{b-84}
Balinski, M. L. (1984) The Hirsch conjecture for dual transportation polyhedra. Math. Oper. Res. 9:4, pp. 629-633
\bibitem{br-74}
Balinski, M., Russakoff, A. (1974) On the assignment polytope. SIAM Rev. 16:4, pp. 516-525
\bibitem{b-11}
Borgwardt, S. (2013) On the Diameter of Partition Polytopes and Vertex-Disjoint Cycle Cover. Math. Program. Ser. A 141:1, pp. 1-20
\bibitem{bbg-13a}
Borgwardt, S., Brieden, A., and Gritzmann, P. (2013) A balanced $k$-means algorithm for weighted point sets. Preprint, arXiv:1308.4004
\bibitem{bfh-14}
Borgwardt, S., Finhold, E., and Hemmecke, R. (2014) On the circuit diameter of dual transportation polyhedra. SIAM J. Discrete Math. Accepted.
\bibitem{blf-14}
Borgwardt, S., De Loera, J. A., and Finhold, E. (2014) Edges vs Circuits: a Hierarchy of Diameters in Polyhedra. Preprint, arXiv:1308.4004
\bibitem{bhs-06}
Brightwell, G., Heuvel, J., and Stougie, L. (2006) A linear bound on the diameter of the transportation polytope. Combinatorica 26:2, pp. 133-139
\bibitem{d-63}
Dantzig, G. (1963) Linear Programming and Extensions. Princeton Univ. Press, Princeton
\bibitem{dk-13}
De Loera, J. A. and  Kim, E. D. (2013) Combinatorics and Geometry of Transportation Polytopes: A thirty-year update. arXiv:1307.0124
\bibitem{dkos-09}
De Loera, J. A., Kim, E. D., Onn, S., and Santos, F.  (2009) Graphs of transportation polytopes. J. Comb. Theory, Ser. A 116:8, pp. 1306-1325
\bibitem{h-41}
Hitchcock, F. K. (1941) The distribution of a product from several sources to numerous localities. J. Math. Phys. 20, pp. 224-230
\bibitem{h-63}
Haley, K. B. (1963) The Multi-Index Problem. Oper. Res. 11:3, pp. 368-379
\bibitem{hk-98}
Holt, F.B.  and Klee, V. (1998) Many polytopes meeting the conjectured Hirsch bound. Discrete Comput. Geom. 20, pp. 1-17
\bibitem{hk-99}
Fritzsche, K. and Holt, F. B. (1999) More polytopes meeting the conjectured Hirsch bound. Discrete Math. 205 (1-3), pp. 77-84
\bibitem{ks-10}
Kim, E. D., Santos, F. (2010) An update on the Hirsch conjecture. Jahresbericht der Deutschen Mathematiker-Vereinigung 112:2, pp. 73-98.
\bibitem{kw-68}
Klee, V., Witzgall, C. (1968) Facets and vertices of transportation polyhedra. In: Dantzig, G.B., Veinott, A.F. (eds.) Math. of the decis.
sci. 1, pp. 257-282. Am. Math. Soc., Providence, Richmond 
\bibitem{kyk-84}
Kovalev, M., Yemelichev, V., and Kratsov, M. (1984) Polytopes, graphs and optimization. Cambridge Univ. Press, Cambridge
\bibitem{mt-93}
Matsui, T., Tamura, S. (1993) Adjacency on Combinatorial Polyhedra. Discrete Appl. Math. 56, pp. 311-321 
\bibitem{n-89}
Naddef, D. (1989) The {H}irsch {C}onjecture is true for $(0,1)$-polytopes. Math. Program. 45:1, pp. 109-110
\bibitem{s-11}
 Santos, F. (2011) A counterexample to the Hirsch conjecture. Annals of Mathematics (Princeton Univ. and Institute for Advanced Study) 176:1, pp. 383-412
\end{thebibliography}
\bibliographystyle{plain}

\end{document}